\documentclass[a4paper,10pt]{article}
\usepackage{authblk}
\usepackage[usenames,dvipsnames,svgnames,table]{xcolor}
\usepackage{graphicx}
\usepackage{multirow}
\usepackage{amssymb}
\usepackage{epstopdf}
\usepackage{pst-node}
\usepackage{pstricks}
\usepackage{xcolor}
\usepackage{pgf}
\usepackage{graphicx}
\usepackage{pict2e}
\usepackage{tikz}
\usepackage{color}
\usepackage{chessboard}
\usepackage{ragged2e}
\usepackage{booktabs}
\usepackage{algorithm}
\usepackage{algorithmic}
\usepackage{chessboard} 
\usepackage{subcaption} 

\usetikzlibrary{arrows}
\tikzstyle{block}=[draw opacity=0.7,line width=1.4cm]

{\usetikzlibrary{arrows,shapes}}

\usepackage[centertags]{amsmath}
\usepackage{amsfonts}
\usepackage{amsthm}
\usepackage{newlfont}
\usepackage{color}
\usepackage{float}
\newtheorem{theorem}{Theorem}[section]
\newtheorem{corollary}[theorem]{Corollary}
\newtheorem{lemma}[theorem]{Lemma}
\newtheorem{proposition}[theorem]{Proposition}
\newtheorem{definition}[theorem]{Definition}
\newtheorem{example}[theorem]{Example}
\newtheorem{remark}[theorem]{Remark}
\newtheorem{conjecture}{Conjecture}

\title{Spectral properties of the $n$-Queens' Graphs}

\author[1,2]{Domingos M. Cardoso}
\author[1,2]{In\^es Ser\^odio Costa}
\author[1,2]{Rui Duarte}

\affil[1]{\small Centro de Investiga\c{c}\~{a}o e Desenvolvimento em Matem\'atica e Aplica\c{c}\~{o}es}
\affil[2]{\small Departamento de Matem\'atica, Universidade de  Aveiro, 3810-193, Aveiro, Portugal.}

\begin{document}

\maketitle

\begin{abstract}
The $n$-Queens' graph, $\mathcal{Q}(n)$, is the graph associated to the $n \times n$ chessboard (a generalization of the classical $8 \times 8$ chessboard), with $n^2$ vertices, each one corresponding to a square of the chessboard. Two vertices of $\mathcal{Q}(n)$ are adjacent if and only if they are in the same row, in the same column or in the same diagonal of the chessboard. After a short overview on the main combinatorial properties of $\mathcal{Q}(n)$, its spectral properties are investigated. First, a lower bound on the least eigenvalue of an arbitrary graph is obtained using clique edge partitions and a sufficient condition for this lower bound be attained is deduced. For the particular case of $\mathcal{Q}(n)$, we prove that for every $n$, its least eigenvalue is not less than $-4$ and it is equal to $-4$ with multiplicity $(n-3)^2$, for every $n \ge 4$. Furthermore, $n-4$ is also an eigenvalue of $\mathcal{Q}(n)$, with multiplicity at least $\frac{n-2}{2}$ when $n$ is even and at least $\frac{n+1}{2}$ when $n$ is odd.  A conjecture about the integer eigenvalues of $\mathcal{Q}(n)$ is presented. We finish this article with an algorithm to determine an equitable partition of the $n$-Queens' graph, $\mathcal{Q}(n)$, for $n \ge 3$, concluding that such equitable partition has $\frac{(\lceil n/2\rceil+1)\lceil n/2\rceil}{2}$ cells.
\end{abstract}

\medskip


\smallskip


\section{Introduction}\label{sec1}

The problem of placing $8$ queens on a chessboard such that no two queens attack each other was first posed in 1848 by a German chess
player \cite{Bezzel}. The German mathematician and physicist Johann Carl Friedrich Gauss (1777--1855) had knowledge of this problem
and found $72$ solutions. However, according to \cite{BellStevens09}, the first to solve the problem by finding all $92$ solutions was
Nauck in \cite{Nauck}, in 1850. As claimed later by Gauss, this number is indeed the total number of solutions. The proof that there is
no more solutions was published by E. Pauls in 1874 \cite{Pauls}.

The $n$-Queens' problem is a generalization of the above problem, consisting of placing $n$ non attacking queens on $n \times n$ chessboard.
In \cite{Pauls} it was also proved that the $n$-Queens' problem has a solution for every $n \ge 4$. Each solution of the $n$-Queens' problem
corresponds to a permutation $\pi = (i_1 \; i_2 \; \dots \; i_n)$, defining the entries $(i_1,1)$, $(i_2,2)$, \ldots, $(i_n,n)$
which are the positions of the $n$ queens on the squares of the $n \times n$ chessboard. Then, we may say that the solutions of the $n$-Queens' problem are a few permutation among the $n!$ permutations of $n$ elements. The corresponding permutation matrices of order $n$ have no two
$1$s both on the main diagonal, both on the second diagonal or both on any other diagonal parallel to
one of these diagonals.\\

In \cite{GentJeffersenNightingale2017} it was proved that a variant of the $n$-Queens' problem (dating to 1850) called
$n$-Queens' completion problem is $\mathbf{NP}$-Complete. In the $n$-Queens' completion problem, assuming that some queens
are already placed, the question is to know how to place the rest of the queens, in case such placement be possible. After
the publication of this result, the $n$-Queens' completion problem has deserved a growing interest on its research. Probably, the motivation is that some researchers believe in the existence of a polynomial-time algorithm to
solve this problem (see \cite{Dmitrii2018}). Therefore, if such an algorithm is found, then the problem that asks whether
$\mathbf{P}$ is equal to $\mathbf{NP}$ is solved. This problem belongs to the list of the Millenium Prize Problems stated
by the Clay Mathematics Institute which awards one million dollars to anyone who finds a solution to any of the seven
problems of the list. So far only one of these problems (the Poincar\'e conjecture) has been solved.

The graph associated to the $n$-Queens' problem, called the $n$-\textit{Queens' graph}, is obtained from the
$n \times n$ chessboard considering its squares as vertices of the graph with two of them adjacent if and only
if they are in the same row, column or diagonal of the chessboard. It is immediate that each solution of the
$n$-Queens' graph corresponds to a maximum stable set (also known as maximum independent vertex set) of the
$n$-Queens' graph. Furthermore, the $n$-Queens' completion problem corresponds to the determination of a maximum
stable set which includes a particular stable set when such maximum stable set exists or the conclusion that
there is no maximum stable set in such conditions.\\

Our focus in this article is neither the $n$-Queens' problem nor the $n$-Queens' completion problem.
The main goal is to study the properties of the $n$-Queens' graph, especially its spectral properties.\\

In Section~\ref{sec2} a short overview on the main combinatorial properties of $\mathcal{Q}(n)$ is presented.
Namely, the stability, clique, chromatic and domination numbers are analyzed.

Section~\ref{sec3} is devoted to the investigation of the spectral properties of $\mathcal{Q}(n)$. A lower
bound on the least eigenvalue of an arbitrary graph is obtained using clique edge partitions and a sufficient
condition for this lower bound be attained is deduced. For the particular case of $\mathcal{Q}(n)$, we prove
that for every $n$, its least eigenvalue is not less than $-4$ and it is equal to $-4$ with multiplicity $(n-3)^2$,
for every $n \ge 4$. Furthermore, $n-4$ is also an eigenvalue of $\mathcal{Q}(n)$, with multiplicity at least
$\frac{n-2}{2}$ when $n$ is even and at least $\frac{n+1}{2}$ when $n$ is odd. We finish this section with a
conjecture about the integer eigenvalues of $\mathcal{Q}(n)$.

In Section~\ref{sec4} an algorithm to determine an equitable partition of the $n$-Queens' graph, $\mathcal{Q}(n)$, with $n \ge 3$, is introduced. This equitable partition has $\frac{(\lceil n/2\rceil+1)\lceil n/2\rceil}{2}$ cells. \\

From now on, the $n \times n$ chessboard is herein denoted by $\mathcal{T}_n$ and the associated graph by $\mathcal{Q}(n)$.
The $n$-Queens' graph, $\mathcal{Q}(n)$, associated to the $n \times n$ chessboard $\mathcal{T}_n$ has $n^2$ vertices,
each one corresponding to a square of the $n \times n$ chessboard. Two vertices of $\mathcal{Q}(n)$ are \textit{adjacent},
that is, linked by an edge if and only if they are in the same row, in the same column or in the same diagonal of the
chessboard. The squares of $\mathcal{T}_n$ and the corresponding vertices in $\mathcal{Q}(n)$ are labeled from the left
to the right and from the top to the bottom. For instance, the squares of $\mathcal{T}_4$ (vertices of $\mathcal{Q}(4)$)
are labelled as depicted in the next Figure.

\begin{figure}[h!]
	\centering
	{	\setchessboard{smallboard}
		\chessboard[
		pgfstyle=
		{[base,at={\pgfpoint{0pt}{-0.4ex}}]text},maxfield=d4,label=false,showmover=false,text= \fontsize{1.2ex}{1.2ex}\bfseries{13},markregions={a1-a1},text= \fontsize{1.2ex}{1.2ex}\bfseries{9},markregions={a2-a2},text= \fontsize{1.2ex}{1.2ex}\bfseries{5},markregions={a3-a3},text= \fontsize{1.2ex}{1.2ex}\bfseries{1},markregions={a4-a4},text= \fontsize{1.2ex}{1.2ex}\bfseries{14},markregions={b1-b1},text= \fontsize{1.2ex}{1.2ex}\bfseries{10},markregions={b2-b2},text= \fontsize{1.2ex}{1.2ex}\bfseries{6},markregions={b3-b3},text= \fontsize{1.2ex}{1.2ex}\bfseries{2},markregions={b4-b4},text= \fontsize{1.2ex}{1.2ex}\bfseries{15},markregions={c1-c1},text= \fontsize{1.2ex}{1.2ex}\bfseries{11},markregions={c2-c2},text= \fontsize{1.2ex}{1.2ex}\bfseries{7},markregions={c3-c3},text= \fontsize{1.2ex}{1.2ex}\bfseries{3},markregions={c4-c4},text= \fontsize{1.2ex}{1.2ex}\bfseries{16},markregions={d1-d1},text= \fontsize{1.2ex}{1.2ex}\bfseries{12},markregions={d2-d2},text= \fontsize{1.2ex}{1.2ex}\bfseries{8},markregions={d3-d3},text= \fontsize{1.2ex}{1.2ex}\bfseries{4},markregions={d4-d4}]}
	\caption{Labeling of $\mathcal{T}_4$.}\label{chessboard4}
\end{figure}

The vertex set and the edge set of $\mathcal{Q}(n)$ are denoted by $V(\mathcal{Q}(n))$ and $E(\mathcal{Q}(n))$, respectively.
The \text{order} of $\mathcal{Q}(n)$ is the cardinality of $V(\mathcal{Q}(n))$, $n^2$, and the \text{size} of $\mathcal{Q}(n)$ is the cardinality of $E(\mathcal{Q}(n))$ and is denoted by $e(\mathcal{Q}(n))$.

\section{A short overview on the main combinatorial properties of $\mathcal{Q}(n)$}\label{sec2}
We start this section by recalling some classical concepts of graph theory.

Given a graph $G$, a \textit{stable} set (resp. \textit{clique}) of $G$ is a vertex subset where every two vertices
are not adjacent (resp. are adjacent). The \textit{stability} (resp. \textit{clique}) number of $G$, $\alpha(G)$ (resp.
$\omega(G)$), is the cardinality of a stable set (resp. clique) of maximum cardinality. A proper coloring of the vertices
of $G$ is a function $\varphi: V(G) \rightarrow C$, where $C$ is a set of colors, such that $\varphi(x) \ne \varphi(y)$
whenever $x$ is adjacent to $y$, that is, $xy \in E(G)$. The chromatic number of $G$, $\chi(G)$, is the minimum cardinality
of $C$ for which there is a proper coloring $\varphi: V(G) \to C$ of the vertices of $G$. From now on, a proper coloring
of the vertices of a graph $G$ is just called vertex coloring of $G$.
A vertex $v \in V(G)$, \textit{dominates} itself and all its neighbors. A vertex set $S \subset V(G)$
is a \textit{dominating} set if every vertex of $G$ is dominated by at least one vertex of $S$. The \textit{domination number}
of a graph $G$, $\gamma(G)$, is the cardinality of a dominating set in $G$ with minimum cardinality.\\

Regarding the \textit{diameter} of $\mathcal{Q}(n)$, $\text{diam}(\mathcal{Q}(n))$, which is the maximum distance
among the distances between every pair of vertices, it is immediate that $\mbox{diam}(\mathcal{Q}(n)) = 2$, for $n>2$
and $\mbox{diam} (\mathcal{Q}(n)) = 1$ for $n \in \{ 1, 2 \}$.

In this section, the size, vertex degrees and average degree, the stability, clique, chromatic and domination numbers of
$\mathcal{Q}(n)$ are analyzed.

\subsection{Size, vertex degrees and average degree}
Taking into account that two vertices of $\mathcal{Q}(n)$ are linked by an edge if and only if they are in the same
row, column or diagonal, it follows that the size of $\mathcal{Q}(n)$ can be obtained by the expression
\begin{eqnarray}
e(\mathcal{Q}(n)) &=& 2(n+1) \binom{n}{2} + 4\left( \binom{2}{2} + \cdots + \binom{n-1}{2} \right) \nonumber\\
                  &=& 2(n+1) \binom{n}{2} + 4 \binom{n}{3} \label{hockey-stick}\\
                  &=& \frac{n(n-1)(5n-1)}{3}. \label{size}
\end{eqnarray}
The expression \eqref{hockey-stick} is obtained using the hockey-stick identity twice.

The \textit{degree} of a vertex $v$, $d_v$, is the number of its \textit{neighbors}, that is, the number of vertices adjacent to $v$.

By the Handshaking Lemma, the \textit{average degree} of a graph $G$ of order $n$ is $\overline{d}_G=\frac{2e(G)}{n}$. From \eqref{size} we get
\begin{equation}
\overline{d}_{\mathcal{Q}(n)}=\frac{2e(\mathcal{Q}(n))}{n^2}=\frac{2(n-1)(5n-1)}{3n}.\label{everage_degree}
\end{equation}

A closed formula (in terms of $n$) for the degrees of the vertices of $\mathcal{Q}(n)$ can be obtained from the structure of these graphs. Before that, we need to introduce some additional notation for the elements of the following partition of $V\left(\mathcal{Q}(n)\right)$.

\begin{enumerate}
\item[$1.$] The \textit{first peripheral vertex subset} $V_1$ is the vertex subset corresponding to the more peripheral
            squares of the board $\mathcal{T}_n$;
\item[$2.$] The \textit{second peripheral vertex subset} $V_2$ is the vertex subset corresponding to the more peripheral squares of the
            board $\mathcal{T}_n$ without considering the vertices in $V_1$;
\item[$\,$] $\vdots$
\item[$\lfloor \frac{n+1}{2} \rfloor.$] The $\lfloor \frac{n+1}{2} \rfloor$-\textit{th peripheral vertex subset}
            $V_{\lfloor \frac{n+1}{2} \rfloor}$ is the vertex subset corresponding to the squares of the board
            $\mathcal{T}_n$, without considering the vertices in
            $V_1 \cup V_2 \cup \dots \cup V_{\lfloor\frac{n+1}{2} \rfloor-1}$.
\end{enumerate}

Clearly, the vertex subsets $V_1, V_2, \dots,  V_{\lfloor  \frac{n+1}{2} \rfloor}$ form a partition of
$V\left(\mathcal{Q}(n)\right)$. For instance, for $n=5$ (see the chessboard $T_5$ depicted in
Figure~\ref{tab_chessboard5}, where the labels of the squares are the labels of the corresponding vertices of
$\mathcal{Q}(n)$), $\lfloor \frac{5+1}{2} \rfloor = 3$ and
\begin{enumerate}
\item $V_1=\{1, 2, 3, 4, 5, 10, 15, 20, 25, 24, 23, 22, 21, 16, 11, 6\}$,
\item $V_2=\{7, 8, 9, 14, 19, 18, 17, 12\}$,
\item $V_3=\{13\}$.
\end{enumerate}

\begin{figure}[h!]
	\centering
	{
	\setchessboard{smallboard}
	\chessboard[
	pgfstyle=
	{[base,at={\pgfpoint{0pt}{-0.4ex}}]text},maxfield=e5,labelbottomformat=\arabic{filelabel},showmover=false,text= \fontsize{1.2ex}{1.2ex}\bfseries{1},markregions={a5-a5},text= \fontsize{1.2ex}{1.2ex}\bfseries{2},markregions={b5-b5},text= \fontsize{1.2ex}{1.2ex}\bfseries{3},markregions={c5-c5},text= \fontsize{1.2ex}{1.2ex}\bfseries{4},markregions={d5-d5},text= \fontsize{1.2ex}{1.2ex}\bfseries{5},markregions={e5-e5},text= \fontsize{1.2ex}{1.2ex}\bfseries{6},markregions={a4-a4},text= \fontsize{1.2ex}{1.2ex}\bfseries{7},markregions={b4-b4},text= \fontsize{1.2ex}{1.2ex}\bfseries{8},markregions={c4-c4},text= \fontsize{1.2ex}{1.2ex}\bfseries{9},markregions={d4-d4},text= \fontsize{1.2ex}{1.2ex}\bfseries{10},markregions={e4-e4},text= \fontsize{1.2ex}{1.2ex}\bfseries{11},markregions={a3-a3},text= \fontsize{1.2ex}{1.2ex}\bfseries{12},markregions={b3-b3},text= \fontsize{1.2ex}{1.2ex}\bfseries{13},markregions={c3-c3},text= \fontsize{1.2ex}{1.2ex}\bfseries{14},markregions={d3-d3},text= \fontsize{1.2ex}{1.2ex}\bfseries{15},markregions={e3-e3},text= \fontsize{1.2ex}{1.2ex}\bfseries{16},markregions={a2-a2},text= \fontsize{1.2ex}{1.2ex}\bfseries{17},markregions={b2-b2},text= \fontsize{1.2ex}{1.2ex}\bfseries{18},markregions={c2-c2},text= \fontsize{1.2ex}{1.2ex}\bfseries{19},markregions={d2-d2},text= \fontsize{1.2ex}{1.2ex}\bfseries{20},markregions={e2-e2},text= \fontsize{1.2ex}{1.2ex}\bfseries{21},markregions={a1-a1},text= \fontsize{1.2ex}{1.2ex}\bfseries{22},markregions={b1-b1},text= \fontsize{1.2ex}{1.2ex}\bfseries{23},markregions={c1-c1},text= \fontsize{1.2ex}{1.2ex}\bfseries{24},markregions={d1-d1},text= \fontsize{1.2ex}{1.2ex}\bfseries{25},markregions={e1-e1}]}
	\caption{The chessboard $T_5$.}\label{tab_chessboard5}
\end{figure}

Regarding the cardinality of those peripheral vertex subsets it follows that
\begin{eqnarray}
\left| V_i \right|  &=& (n-2(i-1))^2 - (n-2(i-1)-2)^2 \nonumber\\
                    &=& \left( (n-2(i-1)) - (n-2(i-1)-2) \right) \left( (n-2(i-1)) + (n-2(i-1)-2) \right) \nonumber\\
                    &=& 2 (2n-4(i-1)-2) = 4 (n-(2i-1)), \;\; i=1 , 2, \dots, \left\lfloor \frac{n+1}{2} \right\rfloor - 1, \label{marca_1}\\
\left| V_{\left\lfloor \frac{n+1}{2} \right\rfloor} \right| &=& \left\{\begin{array}{ll}
                                                                        4, & \hbox{whether $n$ is even}, \label{marca_2}\\
                                                                        1, & \hbox{otherwise.}
                                                                       \end{array}\right.
\end{eqnarray}

For the sake of simplicity of the proof of the next theorem, it is worth to consider the chessboard $\mathcal{T}_n$ as a
$n \times n$ matrix, were the coordinates $(p,q)$ of a square means that it is in the $p$-th row (from the top to bottom)
and $q$-th column (from the left to right).

\begin{theorem}\label{degrees_theorem}
Considering the partition of $V(\mathcal{Q}(n))$ into the peripheral vertex subsets $V_1, V_2, \dots, V_{\lfloor \frac{n+1}{2} \rfloor}$,
the degrees of the vertices of $\mathcal{Q}(n)$ are
\[
d_v = 3(n-1)+2(j-1), \qquad  \text{ for all } v \in V_j, \;  j=1, 2, \dots, \left\lfloor \frac{n+1}{2} \right\rfloor.
\]
\end{theorem}

\begin{proof}
For $v \in V_1$, it is immediate that $d_v=3(n-1)$. On the other hand, for $v \in V_i$, with
$i = \lfloor \frac{n+1}{2} \rfloor$, it is also immediate to conclude that
$$
d_v = \left\{\begin{array}{ll}
               4(n-1) = 2(n-1) +2(\lfloor \frac{n+1}{2} \rfloor -1), & \hbox{if $n$ is odd};\\
               3(n-1) + n-2 = 3(n-1)  +2(\lfloor \frac{n+1}{2} \rfloor -1), & \hbox{otherwise.}
              \end{array}\right.
$$

Let us consider $i$ such that $1 < i < \lfloor \frac{n+1}{2} \rfloor$ and that $v \in V_i$ corresponds to the square of the chessboard $\mathcal{T}_n$ determined by the pair of coordinates $(p,q)$. Then
$(p,q) \in X_{i,i} \cup X_{i,(n-(i-1))} \cup X_{n-(i-1),n-(i-1)} \cup X_{(n-(i-1)),i}$, where
\begin{eqnarray*}
X_{i,i}                &=& \{(i,j) \in[n]^2 \mid   i \le j \le n-i\}\\
X_{i,(n-(i-1))}        &=& \{(j,n-(i-1))  \in[n]^2 \mid  i \le j \le n-i\}\\
X_{n-(i-1),n-(i-1)}    &=& \{(n-(i-1),j) \in[n]^2 \mid i+1 \le j \le n-(i-1)\}\\
X_{n-(i-1),i}          &=& \{(j,i) \in[n]^2 \mid i+1 \le j \le n-(i-1)\}.
\end{eqnarray*}

Assuming $(p,q) \in X_{i,i}$, then $p=i$, $i \le q \le n-i$ and $(p,q)$ has $2(n-1)$ neighbors corresponding to
the squares of the row and column of $\mathcal{T}_n$ whose intersection is $(p,q)$ plus the vertices corresponding
to the diagonal neighbors defined by the following pairs $(x,y)$:
\begin{description}
\item \hspace{1.5cm} $\underbrace{(i+1,q-1), (i+2,q-2), \dots, (i+(q-1),q-(q-1))}_{q-1 \text{ lower right to left diagonal neighbors}}$,
\item \hspace{1.5cm} $\underbrace{(i+1,q+1), (i+2,q+2), \dots, (i+(n-q),q+(n-q))}_{n-q \text{ lower left to right diagonal neighbors}}$,
\item \hspace{1.5cm} $\underbrace{(i-1,q-1), (i-2,q-2), \dots, (i-(i-1),q-(i-1))}_{i-1 \text{ upper right to left diagonal neighbors}}$,
\item \hspace{1.5cm} $\underbrace{(i-1,q+1), (i-2,q+2), \dots, (i-(i-1),q+(i-1))}_{i-1 \text{ upper left to right diagonal neighbors}}$.
\end{description}
Therefore, the total number of neighbors of $(p,q)$ is
$$
d_{(p,q)} = 2(n-1) + q-1 + n-q + 2(i-1) = 3(n-1) + 2(i-1).
$$

For the vertices $(p,q)$ in any other of the subsets $X_{i,(n-(i-1))}$, $X_{(n-(i-1)),(n-(i-1))}$ and $X_{(n-(i-1)),i}$,
by symmetry, the result is the same.
\end{proof}

As immediate consequence, we have the following corollary of Theorem~\ref{degrees_theorem}.

\begin{corollary}
For all vertices $v$ of $\mathcal{Q}(n)$,
\begin{equation}
3(n-1) \le d_v \le  \left\{\begin{array}{ll}
                            4n-5,   & \hbox{if $n$ is even;} \\
                            4n-4, & \hbox{if $n$ is odd.}
                           \end{array}\right. \label{min_max_degree}
\end{equation}
\end{corollary}
Therefore, the \textit{minimum} and \textit{maximum} degrees of $\mathcal{Q}(n)$ are $\delta(\mathcal{Q}(n)) = \mbox{3(n-1)}$ and
$\Delta(\mathcal{Q}(n))=\left\{\begin{array}{ll}
                                 4n-5,   & \hbox{if $n$ is even;} \\
                                 4n-4, & \hbox{if $n$ is odd,}
                                \end{array}\right.$ respectively.

\subsection{Stability, clique, chromatic and domination numbers}
\subsubsection{The stability and clique numbers}
For $n = 2, 3$, it is clear that $\alpha(\mathcal{Q}(n))$ is equal to $1$ and $2$, respectively. Taking into
account that every solution of the $n$-Queens' problem corresponds to a maximum stable set and, as proved in
\cite{Pauls}, the $n$-Queens' problem has a solution for every $n \ge 4$, the next proposition follows.

\begin{proposition}
The stability number of the $n$-Queens' graph is
\[
\alpha(\mathcal{Q}(n)) = \left\{\begin{array}{ll}
                                   1, & \hbox{if } n=2;\\
                                   2, & \hbox{if } n=3;\\
                                   n, & \hbox{if } n \ge 4.
                                  \end{array}\right.
\]
\end{proposition}

The next theorem states the clique number of the $n$-Queens' graph.

\begin{theorem}
The clique number of the $n$-Queens' graph is
\[
\omega(\mathcal{Q}(n)) = \left\{\begin{array}{ll}
                                   4, & \hbox{if } n=2;\\
                                   5, & \hbox{if } n \in \{3, 4\};\\
                                   n, & \hbox{if } n \ge 5.
                                  \end{array}\right.
\]
\end{theorem}

\begin{proof}
The values of the clique number for $n \in \{2, 3, 4\}$ are immediate. For instance, for $n=4$, considering the $5$
vertices labeled by the symbol of a Queen in the squares of the chessboard depicted in Figure~\ref{tab_chessboard3},
it is obviuous that they form a clique and there is no other clique with more than five vertices. Note that there
are no four (resp. three) vertices in the same row or column having two (resp. three) common neighbors outside the row
or column and there are no two vertices in the same row or column with four common neighbors outside the row or column.
\begin{figure}[h!]
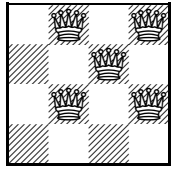

	\centering
	\setchessboard{boardfontsize=15pt,labelfontsize=6pt}
	\chessboard[pgfstyle=4x4,maxfield=d4,setwhite={Qd4, Qc3,Qb2,Qb4,Qd2},label=false,showmover=false]
\caption{Five Queens attacking each other in $\mathcal{T}_4$.}\label{tab_chessboard3}
\end{figure}

Let us assume $n \ge 5$ and that $\mathcal{Q}(n)$ has a clique with $n+1$ vertices. Then, among these vertices (squares of the chessboard) there are two, say $(i,j_1)$ and $(i,j_2)$ (similarly $(i_1,j)$ and $(i_2,j)$) in the same row (resp. column) of the chessboard. The common neighbors to these vertices out of the $i$-th row forming (all together) a maximal
clique are the following.
\begin{enumerate}
\item If $i=1$ (for $i=n$ is similar) and $j_2 = j_1+k$, with $k \ge 1$, those vertices are $(i+k,j_1)$ and $(i+k,j_2)$, plus the vertex
      $(i+k/2,j_1+k/2)$ when $k$ is even.
\item If $1 < i < n$ and $j_2 = j_1+k$, with $k \ge 1$, those vertices are $(i-k,j_1)$ and $(i-k,j_2)$, if $i-k \geq 1$, and the vertex $(i-k/2,j_1+k/2)$ when $k$ is even and $i-k/2 \geq 1$, or $(i+k,j_1)$ and $(i+k,j_2)$, if $i+k \leq n$ plus the vertex $(i+k/2,j_1+ k/2)$ when $k$ is even and $i+k/2 \leq n$. \end{enumerate}

In both cases, there are at most three common neighbors to $(i,j_1)$ and $(i,j_2)$ outside the $i$-th row forming (all together) a maximal clique. Choosing just one vertex, among the common neighbors of $(i,j_1)$ and $(i,j_2)$ outside the $i$-th row, the maximum number of its neighbors belonging to the $i$-th row is three. Therefore, the vertices of any row (or column) form a maximum clique.
\end{proof}

\subsubsection{The chromatic number}
It is well-known that the chromatic number of a graph is not less than its clique number. Therefore,
$\chi(\mathcal{Q}(2)) = 4$ and, since $\omega(\mathcal{Q}(n))=5$ for $n \in \{3, 4, 5\}$,
\begin{equation}
\chi(\mathcal{Q}(n)) = 5,\label{lowerbound}
\end{equation}
as it follows from the vertex colorings presented in Figure~\ref{Q3_Q4_Q5_colorations}, where the set of colors is
$C=\{0,1,2,3,4\}$.

\begin{figure}[h!]
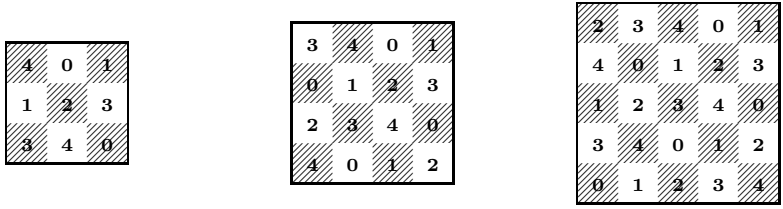

\centering
\begin{subfigure}{0.3\textwidth}
		{	\setchessboard{smallboard}
			\chessboard[
			pgfstyle=
			{[base,at={\pgfpoint{0pt}{-0.4ex}}]text},maxfield=c3,label=false,showmover=false,text= \fontsize{1.2ex}{1.2ex}\bfseries{3},markregions={a1-a1},text= \fontsize{1.2ex}{1.2ex}\bfseries{1},markregions={a2-a2},text= \fontsize{1.2ex}{1.2ex}\bfseries{4},markregions={a3-a3},text= \fontsize{1.2ex}{1.2ex}\bfseries{4},markregions={b1-b1},text= \fontsize{1.2ex}{1.2ex}\bfseries{2},markregions={b2-b2},text= \fontsize{1.2ex}{1.2ex}\bfseries{0},markregions={b3-b3},text= \fontsize{1.2ex}{1.2ex}\bfseries{0},markregions={c1-c1},text= \fontsize{1.2ex}{1.2ex}\bfseries{3},markregions={c2-c2},text= \fontsize{1.2ex}{1.2ex}\bfseries{1},markregions={c3-c3}]}
\end{subfigure}
\begin{subfigure}{0.3\textwidth}
		{	\setchessboard{smallboard}
			\chessboard[
			pgfstyle=
			{[base,at={\pgfpoint{0pt}{-0.4ex}}]text},maxfield=d4,label=false,showmover=false,text= \fontsize{1.2ex}{1.2ex}\bfseries{4},markregions={a1-a1},text= \fontsize{1.2ex}{1.2ex}\bfseries{2},markregions={a2-a2},text= \fontsize{1.2ex}{1.2ex}\bfseries{0},markregions={a3-a3},text= \fontsize{1.2ex}{1.2ex}\bfseries{3},markregions={a4-a4},text= \fontsize{1.2ex}{1.2ex}\bfseries{0},markregions={b1-b1},text= \fontsize{1.2ex}{1.2ex}\bfseries{3},markregions={b2-b2},text= \fontsize{1.2ex}{1.2ex}\bfseries{1},markregions={b3-b3},text= \fontsize{1.2ex}{1.2ex}\bfseries{4},markregions={b4-b4},text= \fontsize{1.2ex}{1.2ex}\bfseries{1},markregions={c1-c1},text= \fontsize{1.2ex}{1.2ex}\bfseries{4},markregions={c2-c2},text= \fontsize{1.2ex}{1.2ex}\bfseries{2},markregions={c3-c3},text= \fontsize{1.2ex}{1.2ex}\bfseries{0},markregions={c4-c4},text= \fontsize{1.2ex}{1.2ex}\bfseries{2},markregions={d1-d1},text= \fontsize{1.2ex}{1.2ex}\bfseries{0},markregions={d2-d2},text= \fontsize{1.2ex}{1.2ex}\bfseries{3},markregions={d3-d3},text= \fontsize{1.2ex}{1.2ex}\bfseries{1},markregions={d4-d4}]}
\end{subfigure}
\begin{subfigure}{0.3\textwidth}
		{	\setchessboard{smallboard}
            \chessboard[
			pgfstyle=
		{[base,at={\pgfpoint{0pt}{-0.4ex}}]text},maxfield=e5,label=false,showmover=false,text= \fontsize{1.2ex}{1.2ex}\bfseries{2},markregions={a5-a5},text= \fontsize{1.2ex}{1.2ex}\bfseries{3},markregions={b5-b5},text= \fontsize{1.2ex}{1.2ex}\bfseries{4},markregions={c5-c5},text= \fontsize{1.2ex}{1.2ex}\bfseries{0},markregions={d5-d5},text= \fontsize{1.2ex}{1.2ex}\bfseries{1},markregions={e5-e5},text= \fontsize{1.2ex}{1.2ex}\bfseries{4},markregions={a4-a4},text= \fontsize{1.2ex}{1.2ex}\bfseries{0},markregions={b4-b4},text= \fontsize{1.2ex}{1.2ex}\bfseries{1},markregions={c4-c4},text= \fontsize{1.2ex}{1.2ex}\bfseries{2},markregions={d4-d4},text= \fontsize{1.2ex}{1.2ex}\bfseries{3},markregions={e4-e4},text= \fontsize{1.2ex}{1.2ex}\bfseries{1},markregions={a3-a3},text= \fontsize{1.2ex}{1.2ex}\bfseries{2},markregions={b3-b3},text= \fontsize{1.2ex}{1.2ex}\bfseries{3},markregions={c3-c3},text= \fontsize{1.2ex}{1.2ex}\bfseries{4},markregions={d3-d3},text= \fontsize{1.2ex}{1.2ex}\bfseries{0},markregions={e3-e3},text= \fontsize{1.2ex}{1.2ex}\bfseries{3},markregions={a2-a2},text= \fontsize{1.2ex}{1.2ex}\bfseries{4},markregions={b2-b2},text= \fontsize{1.2ex}{1.2ex}\bfseries{0},markregions={c2-c2},text= \fontsize{1.2ex}{1.2ex}\bfseries{1},markregions={d2-d2},text= \fontsize{1.2ex}{1.2ex}\bfseries{2},markregions={e2-e2},text= \fontsize{1.2ex}{1.2ex}\bfseries{0},markregions={a1-a1},text= \fontsize{1.2ex}{1.2ex}\bfseries{1},markregions={b1-b1},text= \fontsize{1.2ex}{1.2ex}\bfseries{2},markregions={c1-c1},text= \fontsize{1.2ex}{1.2ex}\bfseries{3},markregions={d1-d1},text= \fontsize{1.2ex}{1.2ex}\bfseries{4},markregions={e1-e1}]}
\end{subfigure}
\caption{A vertex coloring of $\mathcal{Q}(3), \mathcal{Q}(4)$ and $\mathcal{Q}(5)$.}\label{Q3_Q4_Q5_colorations}
\end{figure}

\begin{remark}
Each color class (subset of vertices with the same color) of the vertex coloring of $\mathcal{Q}(5)$ presented in
Figure~\ref{Q3_Q4_Q5_colorations} is a solution for the $5$-Queens' problem. More generally, when $\chi(\mathcal{Q}(n))=n$,
each color class of a vertex coloring of $\mathcal{Q}(n)$ with $n$ colors is a solution for the $n$-Queens' problem.
\end{remark}

The vertex coloring of $\mathcal{Q}(5)$ presented in Figure~\ref{Q3_Q4_Q5_colorations} forms a Latin square with
the particular property that every diagonal has no repeated colors. More general, it is immediate that $\chi(\mathcal{Q}(n))=n$ if and only if there is a Latin square of order $n$ formed by $n$ distinct color symbols, where each color appears once in each of its diagonals. The next theorem gives a sufficient condition for $\chi(\mathcal{Q}(n))=n$. Its proof can be found in \cite{chvatal}. Even so, for the readers convenience, this proof is herein reproduced.

\begin{theorem}
If $n \equiv 1, 5 \mod 6$, then $\chi(\mathcal{Q}(n))=n$.\label{KnutVik_colorations}
\end{theorem}

\begin{proof}
Assume that $n$ is nor even neither divisible by $3$ and consider the vertex coloring proposed in \cite{chvatal},
where the $n \times n$ chessboard is seen as the cartesian product $\{0, 1, \dots, n-1\} \times \{0, 1, \dots, n-1\}$
and the set of colors is $C = \{0, 1, \dots, n-1\}$. Assign to each vertex $(i,j)$ the color $j - 2i \mod n$.

If $(i,j)$ and $(p,q)$ are two vertices with the same color, then $j - 2i \equiv q - 2p \mod n$, which is equivalent to
\begin{equation}
2(p-i) \equiv (q-j) \mod n. \label{same_color}
\end{equation}

\begin{enumerate}
\item If the two vertices are in the same row, that is, $i=p$, then \eqref{same_color} implies $j=q$.
\item If the two vertices are in the same column, that is, $j=q$, then \eqref{same_color} implies $i=p$, since $n$ is odd.
\item If the two vertices are in the same diagonal, then (a) $p-i=q-j$ or (b) $p-i=j-q$ and, from \eqref{same_color},
      \begin{itemize}
      \item[(a)] $p-i=q-j \Rightarrow 2(q-j) \equiv (q-j) \mod n \Rightarrow j = q$ and then $(i,j)=(p,q)$.
      \item[(b)] $p-i=j-q \Rightarrow 2(j-q) \equiv (q-j) \mod n \Rightarrow 3(j-q) \equiv 0 \mod n$ and then $(i,j)=(p,q)$,
      since $n$ is not divisible by $3$. \hfill\qedhere
      \end{itemize}
\end{enumerate}
\end{proof}

It is known that $\chi(\mathcal{Q}(n)) > n$ for $n = 2, 3, 4, 6, 8, 9, 10$ and thus we may say that the sufficient condition, $n \equiv 1, 5 \mod 6$,
for $\chi(\mathcal{Q}(n)) = n$, is also necessary when $n<12$. However, in
general, this condition is not necessary. Counterexamples with $n=12$ and $n=14$ are presented in \cite{Vasquez04}. Figure~\ref{tab_VasquezQ12} shows the former counterexample.

\begin{figure}[h!]
\centering
{\setchessboard{smallboard}
 \chessboard[pgfstyle=
			{[base,at={\pgfpoint{0pt}{-0.4ex}}]text},maxfield=l12,label=false,showmover=false,text= \fontsize{1.5ex}{1.2ex}\bfseries{0},
             markregions={a12-a12},text= \fontsize{1.5ex}{1.2ex}\bfseries{5},markregions={b12-b12},text= \fontsize{1.5ex}{1.2ex}\bfseries{9},
             markregions={c12-c12},text= \fontsize{1.5ex}{1.2ex}\bfseries{6},markregions={d12-d12},text= \fontsize{1.5ex}{1.2ex}\bfseries{3},
             markregions={e12-e12},text= \fontsize{1.5ex}{1.2ex}\bfseries{8},markregions={f12-f12},text= \fontsize{1.5ex}{1.2ex}\bfseries{4},
             markregions={g12-g12},text= \fontsize{1.5ex}{1.2ex}\bfseries{1},markregions={h12-h12},text= \fontsize{1.5ex}{1.2ex}\bfseries{10},
             markregions={i12-i12},text= \fontsize{1.5ex}{1.2ex}\bfseries{11},markregions={j12-j12},text= \fontsize{1.5ex}{1.2ex}\bfseries{7},
             markregions={k12-k12},text= \fontsize{1.5ex}{1.2ex}\bfseries{2},markregions={l12-l12},text= \fontsize{1.5ex}{1.2ex}\bfseries{7},
             markregions={a11-a11},text= \fontsize{1.5ex}{1.2ex}\bfseries{11},markregions={b11-b11},text= \fontsize{1.5ex}{1.2ex}\bfseries{4},
             markregions={c11-c11},text= \fontsize{1.5ex}{1.2ex}\bfseries{2},markregions={d11-d11},text= \fontsize{1.5ex}{1.2ex}\bfseries{1},
             markregions={e11-e11},text= \fontsize{1.5ex}{1.2ex}\bfseries{6},markregions={f11-f11},text= \fontsize{1.5ex}{1.2ex}\bfseries{10},
             markregions={g11-g11},text= \fontsize{1.5ex}{1.2ex}\bfseries{3},markregions={h11-h11},text= \fontsize{1.5ex}{1.2ex}\bfseries{0},
             markregions={i11-i11},text= \fontsize{1.5ex}{1.2ex}\bfseries{8},markregions={j11-j11},text= \fontsize{1.5ex}{1.2ex}\bfseries{9},
             markregions={k11-k11},text= \fontsize{1.5ex}{1.2ex}\bfseries{5},markregions={l11-l11},text= \fontsize{1.5ex}{1.2ex}\bfseries{8},
             markregions={a10-a10},text= \fontsize{1.5ex}{1.2ex}\bfseries{1},markregions={b10-b10},text= \fontsize{1.5ex}{1.2ex}\bfseries{10},
             markregions={c10-c10},text= \fontsize{1.5ex}{1.2ex}\bfseries{9},markregions={d10-d10},text= \fontsize{1.5ex}{1.2ex}\bfseries{5},
             markregions={e10-e10},text= \fontsize{1.5ex}{1.2ex}\bfseries{2},markregions={f10-f10},text= \fontsize{1.5ex}{1.2ex}\bfseries{0},
             markregions={g10-g10},text= \fontsize{1.5ex}{1.2ex}\bfseries{7},markregions={h10-h10},text= \fontsize{1.5ex}{1.2ex}\bfseries{11},
             markregions={i10-i10},text= \fontsize{1.5ex}{1.2ex}\bfseries{6},markregions={j10-j10},text= \fontsize{1.5ex}{1.2ex}\bfseries{3},
             markregions={k10-k10},text= \fontsize{1.5ex}{1.2ex}\bfseries{4},markregions={l10-l10},text= \fontsize{1.5ex}{1.2ex}\bfseries{10},
             markregions={a9-a9},text= \fontsize{1.5ex}{1.2ex}\bfseries{0},markregions={b9-b9},text= \fontsize{1.5ex}{1.2ex}\bfseries{3},
             markregions={c9-c9},text= \fontsize{1.5ex}{1.2ex}\bfseries{8},markregions={d9-d9},text= \fontsize{1.5ex}{1.2ex}\bfseries{7},
             markregions={e9-e9},text= \fontsize{1.5ex}{1.2ex}\bfseries{11},markregions={f9-f9},text= \fontsize{1.5ex}{1.2ex}\bfseries{9},
             markregions={g9-g9},text= \fontsize{1.5ex}{1.2ex}\bfseries{5},markregions={h9-h9},text= \fontsize{1.5ex}{1.2ex}\bfseries{4},
             markregions={i9-i9},text= \fontsize{1.5ex}{1.2ex}\bfseries{1},markregions={j9-j9},text= \fontsize{1.5ex}{1.2ex}\bfseries{2},
             markregions={k9-k9},text= \fontsize{1.5ex}{1.2ex}\bfseries{6},markregions={l9-l9},text= \fontsize{1.5ex}{1.2ex}\bfseries{5},
             markregions={a8-a8},text= \fontsize{1.5ex}{1.2ex}\bfseries{6},markregions={b8-b8},text= \fontsize{1.5ex}{1.2ex}\bfseries{11},
             markregions={c8-c8},text= \fontsize{1.5ex}{1.2ex}\bfseries{4},markregions={d8-d8},text= \fontsize{1.5ex}{1.2ex}\bfseries{2},
             markregions={e8-e8},text= \fontsize{1.5ex}{1.2ex}\bfseries{1},markregions={f8-f8},text= \fontsize{1.5ex}{1.2ex}\bfseries{3},
             markregions={g8-g8},text= \fontsize{1.5ex}{1.2ex}\bfseries{0},markregions={h8-h8},text= \fontsize{1.5ex}{1.2ex}\bfseries{8},
             markregions={i8-i8},text= \fontsize{1.5ex}{1.2ex}\bfseries{9},markregions={j8-j8},text= \fontsize{1.5ex}{1.2ex}\bfseries{10},
             markregions={k8-k8},text= \fontsize{1.5ex}{1.2ex}\bfseries{7},markregions={l8-l8},text= \fontsize{1.5ex}{1.2ex}\bfseries{11},
             markregions={a7-a7},text= \fontsize{1.5ex}{1.2ex}\bfseries{7},markregions={b7-b7},text= \fontsize{1.5ex}{1.2ex}\bfseries{0},
             markregions={c7-c7},text= \fontsize{1.5ex}{1.2ex}\bfseries{1},markregions={d7-d7},text= \fontsize{1.5ex}{1.2ex}\bfseries{10},
             markregions={e7-e7},text= \fontsize{1.5ex}{1.2ex}\bfseries{4},markregions={f7-f7},text= \fontsize{1.5ex}{1.2ex}\bfseries{8},
             markregions={g7-g7},text= \fontsize{1.5ex}{1.2ex}\bfseries{6},markregions={h7-h7},text= \fontsize{1.5ex}{1.2ex}\bfseries{3},
             markregions={i7-i7},text= \fontsize{1.5ex}{1.2ex}\bfseries{2},markregions={j7-j7},text= \fontsize{1.5ex}{1.2ex}\bfseries{5},
             markregions={k7-k7},text= \fontsize{1.5ex}{1.2ex}\bfseries{9},markregions={l7-l7},text= \fontsize{1.5ex}{1.2ex}\bfseries{2},
             markregions={a6-a6},text= \fontsize{1.5ex}{1.2ex}\bfseries{8},markregions={b6-b6},text= \fontsize{1.5ex}{1.2ex}\bfseries{6},
             markregions={c6-c6},text= \fontsize{1.5ex}{1.2ex}\bfseries{3},markregions={d6-d6},text= \fontsize{1.5ex}{1.2ex}\bfseries{9},
             markregions={e6-e6},text= \fontsize{1.5ex}{1.2ex}\bfseries{5},markregions={f6-f6},text= \fontsize{1.5ex}{1.2ex}\bfseries{7},
             markregions={g6-g6},text= \fontsize{1.5ex}{1.2ex}\bfseries{11},markregions={h6-h6},text= \fontsize{1.5ex}{1.2ex}\bfseries{1},
             markregions={i6-i6},text= \fontsize{1.5ex}{1.2ex}\bfseries{10},markregions={j6-j6},text= \fontsize{1.5ex}{1.2ex}\bfseries{4},
             markregions={k6-k6},text= \fontsize{1.5ex}{1.2ex}\bfseries{0},markregions={l6-l6},text= \fontsize{1.5ex}{1.2ex}\bfseries{3},
             markregions={a5-a5},text= \fontsize{1.5ex}{1.2ex}\bfseries{4},markregions={b5-b5},text= \fontsize{1.5ex}{1.2ex}\bfseries{5},
             markregions={c5-c5},text= \fontsize{1.5ex}{1.2ex}\bfseries{0},markregions={d5-d5},text= \fontsize{1.5ex}{1.2ex}\bfseries{11},
             markregions={e5-e5},text= \fontsize{1.5ex}{1.2ex}\bfseries{10},markregions={f5-f5},text= \fontsize{1.5ex}{1.2ex}\bfseries{6},
             markregions={g5-g5},text= \fontsize{1.5ex}{1.2ex}\bfseries{9},markregions={h5-h5},text= \fontsize{1.5ex}{1.2ex}\bfseries{2},
             markregions={i5-i5},text= \fontsize{1.5ex}{1.2ex}\bfseries{7},markregions={j5-j5},text= \fontsize{1.5ex}{1.2ex}\bfseries{8},
             markregions={k5-k5},text= \fontsize{1.5ex}{1.2ex}\bfseries{1},markregions={l5-l5},text= \fontsize{1.5ex}{1.2ex}\bfseries{9},
             markregions={a4-a4},text= \fontsize{1.5ex}{1.2ex}\bfseries{2},markregions={b4-b4},text= \fontsize{1.5ex}{1.2ex}\bfseries{1},
             markregions={c4-c4},text= \fontsize{1.5ex}{1.2ex}\bfseries{10},markregions={d4-d4},text= \fontsize{1.5ex}{1.2ex}\bfseries{4},
             markregions={e4-e4},text= \fontsize{1.5ex}{1.2ex}\bfseries{7},markregions={f4-f4},text= \fontsize{1.5ex}{1.2ex}\bfseries{5},
             markregions={g4-g4},text= \fontsize{1.5ex}{1.2ex}\bfseries{8},markregions={h4-h4},text= \fontsize{1.5ex}{1.2ex}\bfseries{6},
             markregions={i4-i4},text= \fontsize{1.5ex}{1.2ex}\bfseries{3},markregions={j4-j4},text= \fontsize{1.5ex}{1.2ex}\bfseries{0},
             markregions={k4-k4},text= \fontsize{1.5ex}{1.2ex}\bfseries{11},markregions={l4-l4},text= \fontsize{1.5ex}{1.2ex}\bfseries{4},
             markregions={a3-a3},text= \fontsize{1.5ex}{1.2ex}\bfseries{10},markregions={b3-b3},text= \fontsize{1.5ex}{1.2ex}\bfseries{7},
             markregions={c3-c3},text= \fontsize{1.5ex}{1.2ex}\bfseries{11},markregions={d3-d3},text= \fontsize{1.5ex}{1.2ex}\bfseries{0},
             markregions={e3-e3},text= \fontsize{1.5ex}{1.2ex}\bfseries{3},markregions={f3-f3},text= \fontsize{1.5ex}{1.2ex}\bfseries{1},
             markregions={g3-g3},text= \fontsize{1.5ex}{1.2ex}\bfseries{2},markregions={h3-h3},text= \fontsize{1.5ex}{1.2ex}\bfseries{9},
             markregions={i3-i3},text= \fontsize{1.5ex}{1.2ex}\bfseries{5},markregions={j3-j3},text= \fontsize{1.5ex}{1.2ex}\bfseries{6},
             markregions={k3-k3},text= \fontsize{1.5ex}{1.2ex}\bfseries{8},markregions={l3-l3},text= \fontsize{1.5ex}{1.2ex}\bfseries{6},
             markregions={a2-a2},text= \fontsize{1.5ex}{1.2ex}\bfseries{3},markregions={b2-b2},text= \fontsize{1.5ex}{1.2ex}\bfseries{2},
             markregions={c2-c2},text= \fontsize{1.5ex}{1.2ex}\bfseries{5},markregions={d2-d2},text= \fontsize{1.5ex}{1.2ex}\bfseries{8},
             markregions={e2-e2},text= \fontsize{1.5ex}{1.2ex}\bfseries{9},markregions={f2-f2},text= \fontsize{1.5ex}{1.2ex}\bfseries{11},
             markregions={g2-g2},text= \fontsize{1.5ex}{1.2ex}\bfseries{4},markregions={h2-h2},text= \fontsize{1.5ex}{1.2ex}\bfseries{7},
             markregions={i2-i2},text= \fontsize{1.5ex}{1.2ex}\bfseries{0},markregions={j2-j2},text= \fontsize{1.5ex}{1.2ex}\bfseries{1},
             markregions={k2-k2},text= \fontsize{1.5ex}{1.2ex}\bfseries{10},markregions={l2-l2},text= \fontsize{1.5ex}{1.2ex}\bfseries{1},
             markregions={a1-a1},text= \fontsize{1.5ex}{1.2ex}\bfseries{9},markregions={b1-b1},text= \fontsize{1.5ex}{1.2ex}\bfseries{8},
             markregions={c1-c1},text= \fontsize{1.5ex}{1.2ex}\bfseries{7},markregions={d1-d1},text= \fontsize{1.5ex}{1.2ex}\bfseries{6},
             markregions={e1-e1},text= \fontsize{1.5ex}{1.2ex}\bfseries{0},markregions={f1-f1},text= \fontsize{1.5ex}{1.2ex}\bfseries{2},
             markregions={g1-g1},text= \fontsize{1.5ex}{1.2ex}\bfseries{10},markregions={h1-h1},text= \fontsize{1.5ex}{1.2ex}\bfseries{5},
             markregions={i1-i1},text= \fontsize{1.5ex}{1.2ex}\bfseries{4},markregions={j1-j1},text= \fontsize{1.5ex}{1.2ex}\bfseries{11},
             markregions={k1-k1},text= \fontsize{1.5ex}{1.2ex}\bfseries{3},markregions={l1-l1}]}
		\caption{The vertex coloring of $\mathcal{Q}(12)$ with $12$ colors presented in \cite{Vasquez04}.}\label{tab_VasquezQ12}
	\end{figure}

Additional counterexamples for the sufficient condition $n \equiv 1, 5 \mod 6$ being also a necessary condition for
$\chi(\mathcal{Q}(n)) = n$ are referred in \cite{chvatal}, for $n= 15, 16, 18, 20, 21, 22, 24, 28, 32$.

\begin{remark}
In design theory, a Knut Vik design $K$ of order $n$ can be defined (see \cite{HedayatFederer75}) as an $n \times n$
array of elements, chosen from a set of $n$ elements such that (i) with respect to the rows and columns the design is
a Latin square of order $n$ and (ii) for $0 \le j \le n-1$, defining the $j$-th right (resp. left) diagonal of $K$ as
the set of $n$ cells $\{(i,j+i) \mid i=0, 1, \dots, n-1\}$ (resp. $\{(i,j-i-1) \mid i=0, 1, \dots, n-1\}$), where the arithmetic
operations are done modulo $n$, each of these diagonals contains the $n$ elements. Such diagonals are usually called
complete diagonals. In other words, we may say that a Latin square $K$ is a Knut Vik design if every right and left diagonal
of $K$ is complete. The above sets of cells defining right and left diagonals are all broken diagonals, except for $j=0$.\\
In \cite{Hedayat77} the author proves that a Knut Vik design of order $n$ exists if and only if $n$ is not divisible by
$2$ or $3$. If $\mathcal{Q}(n))$ has a coloring defining a Knut Vik design it is immediate that $\chi(\mathcal{Q}(n))=n$.
However, there are values of $n$ for which $\chi(\mathcal{Q}(n))=n$ but in the Latin squares defined by the vertex colorings
not all diagonals are complete and then they are not Knut Vik designs. For instance, in the coloring of $\mathcal{Q}(12)$
depicted in Figure~\ref{tab_VasquezQ12} the color symbol $7$ appears twice in the $5$-th right diagonal (a broken
diagonal) and thus at least one color symbol is missing in the set of n cells of this diagonal.
\end{remark}

\begin{theorem}\label{lower_and_upper_bound_2}
For $n \ge 5$,
\[
n \le \chi(\mathcal{Q}(n)) \le \left\{\begin{array}{ll}
                                    n,   & \hbox{if $n \equiv 1, 5 \mod 6$}; \\
                                    n+3  & \hbox{otherwise.}
                                    \end{array}\right.
\]
\end{theorem}

\begin{proof}
Since $\omega(\mathcal{Q}(n)) = n$,  when $n \ge 5$, the lower bound for $\chi(\mathcal{Q}(n))$ follows.
The upper bound follows from Theorem~\ref{KnutVik_colorations} and taking into account that the chromatic
number of a subgraph of a graph $G$ is not greater than $\chi(G)$.
\end{proof}

As a direct consequence of Theorem~\ref{lower_and_upper_bound_2}, we have the following corollary.

\begin{corollary}
For $n \ge 5$, $\chi(\mathcal{Q}(n)) =n$ when $n \equiv 1,5 \mod 6$ and
\[
\chi(\mathcal{Q}(n)) \in \left\{\begin{array}{ll}
                                     \{n, n+1\},           & \hbox{if } n \equiv 0, 4 \mod 6; \\
                                     \{n, n+1, n+2, n+3\}, & \hbox{if } n \equiv 2 \mod 6; \\
                                     \{n, n+1, n+2\}       & \hbox{if } n \equiv 3 \mod 6.
                                     \end{array}\right.
\]
\end{corollary}

\subsubsection{The domination number}
The study of domination in graphs has its historical roots connected with the classical problems of covering chessboards with
the minimum number of chess pieces, whose mathematical approach goes back to 1862 \cite{Jaenisch1862}. In the 1970s, the
research on the chessboard domination problems was redirected to more general problems of domination in graphs. Since then,
this type of problem has attracted an increasing number of researchers, turning this topic into an intense area of research.\\

The domination number of the Queens' graph, $\gamma\left(\mathcal{Q}(n)\right)$, is the most studied problem
related with this graph. For the first values of $n$, the domination number of $\mathcal{Q}(n)$ is well-known since a long time ago. Indeed, the values of $\gamma\left(\mathcal{Q}(n)\right)$, for $n \in \{ 1, \dots, 8 \}$, presented in
Table~\ref{domination_numbers_for_n le_8}, were given by Rouse Ball in $1892$ \cite{Ball1892}.

\begin{table}[h!]
\centering
\begin{tabular}{c|ccccccccccccc}
             $n$ & 1 & 2 & 3 & 4 & 5 & 6 & 7 & 8 \\\hline
             $\gamma\left(\mathcal{Q}(n)\right)$ & 1 & 1 & 1 & 2 & 3 & 3 & 4 & 5
\end{tabular}
\caption{Values of $\gamma\left(\mathcal{Q}(n)\right)$, for $n \in \{1, \dots, 8\}$.}\label{domination_numbers_for_n le_8}
\end{table}

The following inequalities establish two well-known (lower and upper) bounds for the domination number
\begin{equation}
\frac{n-1}{2} \le \gamma(\mathcal{Q}(n)) \le 2p + q, \label{lower_and_upper_bound}
\end{equation}
where the lower bound holds for every natural number $n$ and the upper bound holds when $n=3p+q$, with
$p \in \mathbb{N}$ and $q \in \{ 0, 1, 2 \}$. The lower bound is due to P.H. Spencer and the upper bound is
due to L. Welch, both cited in \cite{Cockayne1990} as private communications. However, as it is referred in
\cite[Comment added in August 25, 2003]{OstergardWeakley2001}, a proof of the lower bound was published
earlier in \cite{RaghavanVenkatesan1987}. The proof of the upper bound appears in \cite{Cockayne1990}.\\

For particular values of $n$, Weakley \cite{Weakley95} improved the lower bound of \eqref{lower_and_upper_bound} as follows.

\begin{theorem}
For all $k \in \mathbb{N} \cup \{0\}$, $2k+1 \le \gamma(Q(4k+1))$. \label{ThWeakley95}
\end{theorem}

According to Tables~\ref{domination_numbers_for_n le_8} and \ref{domination_numbers}, it is known that
$\gamma(Q(4k+1))=2k+1$, for $k \in \{1, \dots, 32 \}$. However, the problem of knowing for each nonnegative integer $k$ whether the lower bound in Theorem~\ref{ThWeakley95} is sharp or not remains open.

\begin{conjecture}\cite{Weakley95}\label{conjetura_1}
For all $k \in \mathbb{N} \cup \{0\}$, $2k+1 = \gamma(Q(4k+1))$.
\end{conjecture}

Another open problem, herein posed as a conjecture, was proposed in \cite[Problem 1]{Cockayne1990}.

\begin{conjecture}\label{conjetura_2}
For all $n \in \mathbb{N}$, $\gamma(\mathcal{Q}(n)) \le \gamma(\mathcal{Q}(n+1))$.
\end{conjecture}

However, we have the following inequality.

\begin{proposition}\label{dominating_number_inequality}
For $n \in \mathbb{N}$, $\gamma(\mathcal{Q}(n+1)) \le \gamma(\mathcal{Q}(n))+1$.
\end{proposition}

\begin{proof}
Let $S \subset V(\mathcal{Q}(n))$ be a minimum dominating set of $\mathcal{Q}(n)$. Considering that the vertices of
$V(\mathcal{Q}(n))$ occupy the first $n$ rows and first $n$ columns of $\mathcal{Q}(n+1)$ and $S'$ is the vertex
subset of $\mathcal{Q}(n+1)$ corresponding to $S$, it is immediate that $S' \cup \{(n+1,n+1)\}$ is a dominating set
of $\mathcal{Q}(n+1)$ and thus $\gamma(\mathcal{Q}(n+1)) \le \gamma(\mathcal{Q}(n))+1$.
\end{proof}

Table~\ref{domination_numbers} presents the values of $\gamma(\mathcal{Q}(n))$ computed  in \cite{OstergardWeakley2001},
for $n \ge 9$. Almost all of these values satisfy the condition $\gamma(\mathcal{Q}(n))=\lceil \frac{n}{2} \rceil$. In Table~\ref{domination_numbers} we display all known values or $\gamma(\mathcal{Q}(9i+j))$, for $j=0, 1, \dots, 8$ and $i=1, 2, \dots, 14$.

\begin{table}[h!]
\centering
\begin{tabular}{|c|c|c|c|c|c|c|c|c|c|}\hline
             $i \setminus j$  & 0 & 1 & 2 & 3 & 4 & 5 & 6 & 7 & 8 \\ \hline
             1    & 5 & 5 & 5 & 6 & 7 & 8 & 9 & 9 & 9 \\ \hline
             2    & 9 &10 &   &11 &   &12 &   &13 &   \\ \hline
             3    &14 &   &15 &15 &16 &   &17 &   &   \\ \hline
             4    &   &19 &   &20 &   &21 &   &   &   \\ \hline
             5    &23 &   &   &   &25 &   &   &   &27 \\ \hline
             6    &   &   &   &29 &   &   &   &31 &   \\ \hline
             7    &   &   &33 &   &   &   &35 &   &36 \\ \hline
             8    &   &37 &   &   &   &39 &   &   &   \\ \hline
             9    &41 &   &   &   &43 &   &   &   &45 \\ \hline
             10   &   &46 &   &47 &   &   &   &49 &   \\ \hline
             11   &   &   &51 &   &   &   &53 &   &   \\ \hline
             12   &   &55 &   &   &   &57 &   &58 &   \\ \hline
             13   &59 &   &   &   &61 &   &   &   &63 \\ \hline
             14   &   &   &   &65 &65 &66 &   &   &   \\ \hline
\end{tabular}
\caption{Known values of $\gamma\left(\mathcal{Q}(9i+j)\right)$, for $j=0, 1, \dots, 8$ and for $i=1, 2, \dots, 14 $.}\label{domination_numbers}
\end{table}

Some values presented in Table~\ref{domination_numbers} were published by other authors before \cite{OstergardWeakley2001}, such as in \cite{Burger98, BurgerCockayneMynhardt94, GibbonsWebb97}.\\

The domination number can be determined by solving a binary ($0$-$1$) linear programming problem, as it is stated by
the following theorem.

\begin{theorem}
Let $G$ be a graph with adjacency matrix $\mathbf{A}$ and consider the binary linear programming problem
\begin{eqnarray}
\max \; \; \eta(\mathbf{x})                               &  =  & \sum_{j \in V(G)}{x_j}, \nonumber \\
s.t. \; \; \left(\mathbf{A}+\mathbf{I_n}\right)\mathbf{x} & \le & \mathbf{d}, \label{blp} \\
                                               \mathbf{x} & \in & \{0, 1\}^{n}, \nonumber
\end{eqnarray}
where $\mathbf{I_n}$ is the identity matrix of order $n$ and $\mathbf{d}$ is a vector whose entries are the
degrees of the vertices of $G$. If $\mathbf{x}^* \in \{0, 1\}^{n}$ is an optimal solution, denoting the components
of $\mathbf{x}^*$ by $x^*_j$, for all $j \in V(G)$,
$D^* = \{j \mid x^*_j=0\}$
is a minimum dominating set for $G$ and $\gamma(G)=n-\eta(\mathbf{x}^*)$.
\end{theorem}

\begin{proof}
The inequalities \eqref{blp} guarantee that the vertex set defined by the null components of every feasible
solution $\mathbf{x}$ is a dominating set. Note that for each vertex $j$, the variables $x_j$ in the associated
inequality correspond to vertices belonging to the closed neighborhood of $j$ and there is at least one of these
variables $x_r=0$, otherwise $\mathbf{x}$ is not feasible. As a consequence, assuming that $D=\{i \mid x_i=0\}$, for
every vertex $j$ there is at least one vertex in $D$ adjacent to $j$.
If $\mathbf{x}^*$ is an optimal solution of the binary linear programming problem,
then the dominating vertex subset $D^*=\{i \mid x^*_i=0\}$ is such that $|D^*|=n-\eta(\mathbf{x}^*)$.
Let us assume that there is a vertex dominating set $D'$ such that $|D'| < |D^*|$. Then, defining $\mathbf{x}'$
with components $x'_{i} = \left\{\begin{array}{ll}
                                   0, & \hbox{if } i \in D';\\
                                   1, & \hbox{otherwise}
                                  \end{array}\right.$,
it is clear that $\mathbf{x}'$ is a feasible solution
for the binary linear programming problem for which $\eta(\mathbf{x}')=n-|D'|$. Since
$|D'| < |D^*|$, it follows that $\eta(\mathbf{x}^*) = n - |D^*| < n - |D'| = \eta(\mathbf{x}')$, which is a
contradiction. Therefore, $\gamma(G) = |D^*| = n - \eta(\mathbf{x}^*)$.
\end{proof}

As a by-product, in the case of the Queens' graph $\mathcal{Q}(n)$, we have the next corollary, where each variable $x_{ij}$ corresponds to a vertex of $\mathcal{Q}(n)$ with coordinates $(i,j)$ (the square of $\mathcal{T}_n$ in line $i$ from top to bottom and column $j$ from left to right).

\begin{corollary}
Consider the binary linear programming problem
\begin{eqnarray}
\max \; \; \eta(\mathbf{x})                             &  =  & \sum_{(i,j) \in [n]^2}{x_{ij}}, \label{blp1}\\
s.t. \; \; \left(\mathbf{A}+\mathbf{I_n}\right)\mathbf{x} & \le & \mathbf{d}, \label{blp2}\\
                                             \mathbf{x} & \in & \{0, 1\}^{n^2}, \label{blp3}
\end{eqnarray}
where $\mathbf{A}$ is the adjacency matrix of $\mathcal{Q}(n)$, $\mathbf{I_n}$ is the identity matrix of order $n$ and
$\mathbf{d}$ is a vector whose entries are the degrees of the vertices $(p,q)$ of $\mathcal{Q}(n)$.
If $\mathbf{x}^* \in \{0, 1\}^{n^2}$ is an optimal solution, denoting the components
of $\mathbf{x}^*$ by $x^*_{pq}$, for all $(p,q) \in [n]^2$, $D^* = \{(i,j) \mid x^*_{ij}=0\}$
is a minimum dominating set of $\mathcal{Q}(n)$ and $\gamma(\mathcal{Q}(n))=n^2-\eta(\mathbf{x}^*)$.
\end{corollary}

It is computational hard to solve the binary linear programming problem \eqref{blp1} -- \eqref{blp3}, even for not very large values of $n$. So far, by solving  the binary linear programming problem \eqref{blp1} -- \eqref{blp3}, we are able to confirm many of the presented values in Table~\ref{domination_numbers}. However, the domination number problem of the $n$-Queens' graph remains open even for several integers $n \le 131$. For instance, for $n=20$, from the lower bound in \eqref{lower_and_upper_bound} and Proposition~\ref{dominating_number_inequality} it follows that $\gamma(\mathcal{Q}(n)) \in \{10, 11\}$. However, as far as we know, its precise value is unknown.\\

We observe that for all values displayed in Table~\ref{domination_numbers}, $\lfloor \frac{n+1}{2}\rfloor-1 \le \gamma(\mathcal{Q}(n)) \le \lfloor \frac{n+1}{2}\rfloor+1$, if $n \le 16$, and $\gamma\left(\mathcal{Q}(n)\right)=\lfloor \frac{n+1}{2}\rfloor$, if $n \ge 17$.

\section{Spectral properties of $\mathcal{Q}(n)$}\label{sec3}
The eigenvalues of the adjacency matrix of a graph $G$ are simply called the eigenvalues of $G$ and its spectrum
is the multiset $\sigma(G)=\left\{\mu_1^{[m_1]}, \mu_2^{[m_2]}, \ldots, \mu_p^{[m_p]}\right\}$, where $\mu_1 > \mu_2 > \cdots > \mu_p$
are the $p$ distinct eigenvalues (indexed in decreasing order) and $m_i$ (also denoted by $m(\mu_i)$) is the multiplicity of the
eigenvalue $\mu_i$, for $i=1,2, \ldots, p$. When $m_j=1$, we just write $\mu_j$. When necessary, we write $\mu_j(G)$ to clarify that $\mu_j$ is the $j$-th distinct eigenvalue of the graph $G$.
We denote by $\mathcal{E}_G(\mu)$ the eigenspace associated with an eigenvalue $\mu$ of $G$.\\

As it is well-known, the largest eigenvalue of a graph $G$ is between its average degree, $\overline{d}_G$, and its maximum degree,
$\Delta(G)$, and it is equal to both if and only if $G$ is regular. Therefore, from \eqref{everage_degree} and \eqref{min_max_degree}, we may conclude that

\begin{equation}\label{bounds_on_index}
\frac{2(n-1)(5n-1)}{3n} \le \mu_1(\mathcal{Q}(n)) \le \left\{\begin{array}{ll}
                                                                     4n-5 & \hbox{if $n$ is even};\\
                                                                     4n-4 & \hbox{if $n$ is odd}.
                                                              \end{array}\right.
\end{equation}

The characteristic polynomials and spectra of $\mathcal{Q}(2)$ and $\mathcal{Q}(3)$ are the following.
\begin{equation*}
\begin{split}
q_2(x) &= (x+1)^3(x-3) \qquad \hbox{ and } \qquad  \sigma(\mathcal{Q}(2)) = \left\{3,-1^{[3]}\right\}; \\
q_3(x) &= (-1 + x)(1 + x)^2(-8 - 5 x + x^2)(-1 + 2 x + x^2)^2 \qquad \hbox{ and } \\
       &  \qquad \sigma(\mathcal{Q}(3)) = \left\{\frac{5 + \sqrt{57}}{2}, 1, (-1+\sqrt{2})^{[2]}, -1^{[2]},\frac{5 - \sqrt{57}}{2}, (-1-\sqrt{2})^{[2]}\right\}.
\end{split}
\end{equation*}

\subsection{A lower bound on the least eigenvalue of a graph}
We start the study of spectral properties introducing a general result about lower bounds on the least eigenvalues of graphs. Before that it is worth to define a couple of parameters associated with edge clique partitions. The edge clique
partitions (ECP for short) were introduced in \cite{Orlin1977}, where the \textit{content} of a graph
$G$ was defined as the minimum number of edge disjoint cliques whose union includes all the edges of $G$
and denoted by $C(G)$. Such minimum  ECP is called in \cite{Orlin1977} \textit{content decomposition} of $G$ and, as proved there, in general, its determination is $\mathbf{NP}$-Complete.

\begin{definition}(Clique degree and maximum clique degree)\label{new_graph_parameters}
Consider a graph $G$ and an ECP, $P=\{ E_i \mid i \in I \}$ and then $V_i=V(G[E_i])$ is a clique of $G$
for every $i \in I$. For any $v \in V(G)$, the clique degree of $v$ relative to $P$, denoted $m_v(P)$, is the number of cliques $V_i$ containing the vertex $v$, and the maximum clique degree of $G$ relative to P, denoted $m_G(P)$, is the maximum of clique degrees of the vertices of $G$ relative to $P$.
\end{definition}

From Definition~\ref{new_graph_parameters}, considering an ECP, $P=\{ E_i \mid i \in I \}$, the parameters $m_v(P)$ and $m_G(P)$ can be expressed as follows.
\begin{eqnarray}
m_v(P) &=& | \{ i \in I \mid v \in V(G[E_i]) \} | \qquad \forall v \in V(G); \label{parameter_mv}\\
m_G(P)   &=& \max \{ m_v \mid v \in V(G) \}. \label{parameter_m}
\end{eqnarray}

\begin{remark}\label{content_numbe_lower_bound_remark}
It is clear that if $P$ is an ECP of $G$, then $m_G(P)$ is not greater than $|P|$. In particular, if $P$ is a content decomposition of $G$, then $m_G(P) \le C(G)$.
\end{remark}

Using the above defined graph parameters, the next theorem states a general lower bound on the least eigenvalue of a graph. As it will be seen later, there are extremal graphs for which this lower bound is attained, namely the $n$-Queens' graph $\mathcal{Q}(n)$, with $n \ge 4$, considering an ECP whose parts are the edge subsets associated with the $n$ columns, $n$ rows, $2n-3$ left to right diagonals and $2n-3$ right to left diagonals of the corresponding $n \times n$ chessboard $\mathcal{T}_n$.

\begin{theorem}\label{general_lower_bound}
Let  $P=\{ E_i \mid i \in I \}$ be an ECP of a graph $G$, $m=m_G(P)$ and $m_v=m_v(P)$ for every $v \in V(G)$. Then
\begin{enumerate}
\item If $\mu$ is an eigenvalue of $G$, then $\mu \ge -m$. \label{marca}
\item $-m$ is an eigenvalue of $G$ iff there exists a vector $X \ne \mathbf{0}$ such that
      \begin{enumerate}
      \item $\sum \limits_{j \in V(G[E_i])} x_j = 0$, for every $i \in I$ and \label{cond1}
      \item $\forall v \in V(G) \;\; x_v=0$ whenever $m \ne m_v$. \label{cond2}
      \end{enumerate}\label{marca2}
      In the positive case, $X$ is an eigenvector associated with the eigenvalue $-m$.
\end{enumerate}
 \end{theorem}

\begin{proof}
Consider \textbf{A} as the adjacency matrix of $G$.
\begin{enumerate}
\item Let $X$ be an eigenvector of $G$ associated with an eigenvalue $\mu$. Then
	  \begin{align*}
		(\mu+m) \| X \|^2 & = X^T (\mu X + m X) = X^T (\mathbf{A} + m I_n) X \\
		& = \sum_{i \in I} \sum_{uv \in E_i} \left( 2 x_u x_v \right) + m \| X \|^2\\
		& = \sum_{i \in I} \left( \sum_{v \in V(G[E_i])} x_v \right)^2
		- \sum \limits_{v \in V(G)}{m_vx_v^2} + m \| X \|^2\\
		& = \sum_{i \in I} \left( \sum_{v \in V(G[E_i])} x_v \right)^2 +
		\sum \limits_{v \in V(G)} (m-m_v)x_v^2 \ge 0.
	  \end{align*}\label{lower_boumd}
\item If $-m$ is an eigenvalue of $G$, then, from the proof of item \ref{lower_boumd}, equalities \ref{cond1} and \ref{cond2} follow.
      Conversely, if there exists a vector $X \ne \mathbf{0}$ for which \ref{cond1} and \ref{cond2} hold, then
      \begin{eqnarray*}
       0 &=& \sum_{i \in I} \left( \sum_{v \in V(G[E_i])} x_v \right)^2 + \sum \limits_{v \in V(G)} (m-m_v)x_v^2\\
         &=& \sum_{i \in I} \left( \sum_{v \in V(G[E_i])} x_v \right)^2 - \sum \limits_{v \in V(G)}{m_vx_v^2} + m \| X \|^2\\
         &=& \sum_{i \in I} \sum_{uv \in E_i} \left( 2 x_u x_v \right) + m \| X \|^2\\
         &=& X^T \mathbf{A} X + m \| X \|^2.
      \end{eqnarray*}
      Assuming that $\mu$ is the least eigenvalue of $G$, $-m = \frac{X^T \mathbf{A}  X}{\| X \|^2} \ge \mu$.
      Since $\mu \ge -m$ (by item \ref{lower_boumd}), we obtain $-m \ge \mu \ge -m$ and thus $\mu = -m$.
      In the positive case, it is clear that $X$ is an eigenvector associated to the eigenvalue $-m$.
      \qedhere
\end{enumerate}
\end{proof}

Theorem~\ref{general_lower_bound} allows to obtain a spectral lower bound for the content number of a graph.

\begin{corollary}
Let $\mu$ be the least eigenvalue of a graph $G$. Then $-\mu \le C(G)$.
\end{corollary}

\begin{proof}
If $P$ is a content decomposition of $G$, then, according to Remark~\ref{content_numbe_lower_bound_remark}, $m_G(P) \le C(G)$. By Theorem~\ref{general_lower_bound} $-m_G(P) \le \mu$ and so $-\mu \le C(G)$.
\end{proof}

The following corollary is a direct consequence of Theorem~\ref{general_lower_bound}.

\begin{corollary}
Let $G$ be a graph of order $n$ and let $X$ be a vector of $\mathbb{R}^n \setminus \{\mathbf{0}\}$. Then
$X \in \mathcal{E}_{G}(-m)$ iff the conditions \ref{cond1} and \ref{cond2} of Theorem~\ref{general_lower_bound} hold.
\end{corollary}

\subsection{Eigenvalues and eigenvectors of the $n$-Queens' graph}

From now on we consider the particular case of the $n$-Queens' graph $\mathcal{Q}(n)$.

\begin{theorem} \label{nessufcond}
Let $n\in\mathbb{N}$ such that $n \ge 4$.
\begin{enumerate}
\item If $\mu$ is an eigenvalue of $\mathcal{Q}(n)$, then $\mu \ge -4$.
\item $-4 \in \sigma(\mathcal{Q}(n))$ iff there exists a vector $X \in \mathbb{R}^{n^2}\setminus \{\mathbf{0}\}$
      such that
      \begin{enumerate}
      \item $\sum \limits_{j=1}^n {x_{(k,j)}} = 0$ and $\sum \limits_{i=1}^n {x_{(i,k)}} = 0$, for every $k \in [n]$, \label{nessuf_1}
      \item $\sum \limits_{i+j=k+2}{x_{(i,j)}} = 0$, for every $k \in [2n-3]$, \label{nessuf_2}
      \item $\sum \limits_{i-j=k+1-n} x_{(i,j)} = 0$, for every $k \in [2n-3]$, \label{nessuf_3}
      \item $x_{(1,1)} = x_{(1,n)} = x_{(n,1)} = x_{(n,n)} = 0$. \label{nessuf_4}
      \end{enumerate}
      In the positive case, $X$ is an eigenvector associated with the eigenvalue $-4$.
\end{enumerate}
\end{theorem}

\begin{proof}
The proof follows taking into account that the summations \ref{nessuf_1}-\ref{nessuf_3} correspond to the
summations \ref{cond1} in Theorem~\ref{general_lower_bound}. Here, the cliques obtained from the ECP, $P$, of $\mathcal{Q}(n)$ are the cliques with vertices associated with each of the $n$ columns, $n$ rows, $2n-3$ left to right diagonals and $2n-3$ right to left diagonals. Denoting the vertices of $\mathcal{Q}(n)$ by their coordinates $(i,j)$ in the corresponding chessboard $\mathcal{T}_n$,
$
m_{(i,j)}(P) = \left\{\begin{array}{ll}
                        3, & \hbox{if } (i,j) \in \{ 1, n\}^2;\\
                        4, & \hbox{otherwise}
                 \end{array}\right.
$
and thus $m_{\mathcal{Q}(n)}(P)=4$. Therefore, the equalities \ref{nessuf_4} correspond to the conditions \ref{cond2} in Theorem~\ref{general_lower_bound}.
\end{proof}

As a consequence of the previous Theorem we have the following result.

\begin{corollary}\label{non-main_condition1}
Let $n \geq 4$ and $X \in \mathbb{R}^{n^2}$. Then $X \in \mathcal{E}_{\mathcal{Q}(n)}(-4)$ if and only if the conditions \ref{nessuf_1}-\ref{nessuf_4} of Theorem~\ref{nessufcond} hold.
\end{corollary}

From the computations, for several values of $n \ge 4$, it seems that $n-4$ and $-4$ are eigenvalues of $\mathcal{Q}(n)$ with multiplicities
$$
m(n-4) = \left\{\begin{array}{ll}
                  (n-2)/2, & \hbox{if $n$ is even;} \\
                  (n+1)/2, & \hbox{if $n$ is odd,}
            \end{array}
          \right.
$$
and $m(-4)=(n-3)^2$. It also seems that there are other integer eigenvalues when $n$ is odd, and that they are all simple.

\begin{table}[h!]
\centering
\begin{tabular}{|c|c|}\hline
              $n$                   &  Distinct integer eigenvalues      \\ \hline
               3                    &  1, -1                             \\ \hline
               4                    &  0, -4                             \\ \hline
               5                    &  1, 0, -3, -4                      \\ \hline
               6                    &  2, -4                             \\ \hline
               7                    &  3, 2, 1, -2, -3, -4               \\ \hline
               8                    &  4, -4                             \\ \hline
               9                    &  5, 4, 3, 2, -1, -2, -3, -4        \\ \hline
               10                   &  6, -4                             \\ \hline
               11                   &  7, 6, 5, 4, 3, 0, -1, -2, -3, -4  \\ \hline
 \end{tabular}
\caption{Distinct integer eigenvalues of $\mathcal{Q}(n)$, when $3 \le n \le 11$.}\label{computed_igenvalues}
\end{table}

In what follows we will see that, for $n \ge 4$, $-4$ is an eigenvalue of $\mathcal{Q}(n)$ with multiplicity $(n-3)^2$.
From Corollary~\ref{non-main_condition1} it follows that the multiplicity of $-4$ as an eigenvalue of $\mathcal{Q}(n)$
coincides with the corank of the coefficient matrix of the system of $6n-2$ linear equations
\ref{nessuf_1}-\ref{nessuf_4}. Therefore, to say that $m(-4)=(n-3)^2$ is equivalent to say that the rank of the coefficient
matrix of the system of $6n$ linear equations \ref{nessuf_1}-\ref{nessuf_4} is $6n - 9$ (since $n^2 - 6n + 9 = (n-3)^2$).

Before we continue, we need to define a family of vectors
$$
\mathcal{F}_n = \{X_n^{(a,b)} \in \mathbb{R}^{n^2} \mid (a,b) \in [n-3]^2 \}
$$
where $X_4$ is the vector presented in Table~\ref{vector_X4} and $X_n^{(a,b)}$ is the vector defined by

\begin{table}[h!]
\[ \scriptsize
\begin{array}{|c|c|c|c|} \hline
		0          & \textbf{1} & \textbf{-1}& 0           \\ \hline
		\textbf{-1}& 0          & 0          & \textbf{1}  \\ \hline
		\textbf{1} & 0          & 0          & \textbf{-1} \\ \hline
		0          &\textbf{-1} & \textbf{1} & 0           \\ \hline
\end{array}
\]
\caption{The vector $X_4$.}\label{vector_X4}
\end{table}

\begin{equation}\label{ev_components}
\big[X_n^{(a,b)}\big]_{(i,j)} = \begin{cases}
				                \big[X_4\big]_{(i-a+1,j-b+1)}, & \text{ if } (i,j)\in A \times B;\\
				                0,                             & \text{otherwise,}
                                \end{cases}
\end{equation}
with $A = \{a,a+1,a+2,a+3\}$ and $B = \{b,b+1,b+2,b+3\}$. For instance, for $n=5$, $\mathcal{F}_5$ is the family of four vectors depicted in Table~\ref{vetors_11-12-21-22}.

\begin{table}[h!]
\[ \scriptsize \!\!\!\!\!\!
\begin{array}{|c|c|c|c|c|} \hline
		0          & \textbf{1} & \textbf{-1}& 0          & 0 \\ \hline
		\textbf{-1}& 0          & 0          & \textbf{1} & 0 \\ \hline
		\textbf{1} & 0          & 0          & \textbf{-1}& 0 \\ \hline
		0          &\textbf{-1} & \textbf{1} & 0          & 0 \\ \hline
		0          & 0          & 0          & 0          & 0 \\ \hline
\end{array}
\quad
\begin{array}{|c|c|c|c|c|} \hline
		0 & 0          & \textbf{1} & \textbf{-1}& 0 \\ \hline
		0 & \textbf{-1}& 0          & 0          & \textbf{1} \\ \hline
	    0 & \textbf{1} & 0          & 0          & \textbf{-1} \\ \hline
		0 & 0          & \textbf{-1}& \textbf{1} & 0 \\ \hline
		0 & 0          & 0          & 0          & 0 \\ \hline
\end{array}
\quad
\begin{array}{|c|c|c|c|c|}\hline
0          & 0          & 0          & 0          & 0 \\ \hline
0          & \textbf{1} & \textbf{-1}& 0          & 0\\ \hline
\textbf{-1}& 0          & 0          & \textbf{1} & 0 \\ \hline
\textbf{1} & 0          & 0          & \textbf{-1}& 0 \\ \hline
0          & \textbf{-1}& \textbf{1} & 0          & 0 \\ \hline
\end{array}
\quad
\begin{array}{|c|c|c|c|c|}\hline
0 & 0          & 0          & 0          & 0  \\ \hline
0 & 0          & \textbf{1} & \textbf{-1}& 0 \\ \hline
0 & \textbf{-1}& 0          & 0          & \textbf{1} \\ \hline
0 & \textbf{1} & 0          & 0          & \textbf{-1} \\ \hline
0 & 0          & \textbf{-1}& \textbf{1} & 0 \\ \hline
\end{array}
\]
\caption{The vectors $X_5^{(1,1)}$, $X_5^{(1,2)}$, $X_5^{(2,1)}$, and $X_5^{(2,2)}$.} \label{vetors_11-12-21-22}
\end{table}

\begin{theorem}\label{th-4eigenvalue}
$-4$ is an eigenvalue of $\mathcal{Q}(n)$ with multiplicity $(n-3)^2$ and $\mathcal{F}_n$ is a basis for
$\mathcal{E}_{\mathcal{Q}(n)} (-4)$.
\end{theorem}

\begin{proof}
First, note that every element of $\mathcal{F}_n$ belongs to $\mathcal{E}_{\mathcal{Q}(n)} (-4)$. Indeed,
if $X = \left( x_{(i,j)} \right) \in \mathcal{F}_n$, then the conditions \ref{nessuf_1}-\ref{nessuf_4} of
Theorem~\ref{nessufcond} hold and hence by Corollary~\ref{non-main_condition1} $X \in \mathcal{E}_{\mathcal{Q}(n)} (-4)$.

Second, $\mathcal{F}_n$ is linearly independent and so $\dim \mathcal{E}_{\mathcal{Q}(n)} (-4) \geq (n-3)^2$. For otherwise there would be scalars $\alpha_{1,1}, \ldots, \alpha_{n-3,n-3} \in \mathbb{R}$, not all equal to zero, such that
\begin{equation} \label{comblin}
\alpha_{1,1} X_n^{(1,1)} + \cdots + \alpha_{n-3,n-3} X_n^{(n-3,n-3)} = \mathbf{0}.
\end{equation}

Let $(n-3)(a-1)+b$ be the smallest integer such that $\alpha_{a,b} \neq 0$. Since by \eqref{ev_components} $\big[X_n^{(a,b)}\big]_{(a,b+1)} = \big[X_4\big]_{(1,2)}=1$, the entry $(a,b+1)$ of $\alpha_{a,b}\big[X_n^{(a,b)}\big]$ is $\alpha_{a,b}$.  Consider any other vector $\big[X_n^{(a',b')}\big]$ such that $(n-3)(a'-1)+b'>(n-3)(a-1)+b$ which implies (i) $a'>a$ or (ii) $a'=a$ and $b'>b$.
Denoting $A'=\{a', \dots, a'+3\}$ and $B'=\{b', \dots, b'+3\},$ taking in to account \eqref{ev_components}, we may conclude the following.
\begin{itemize}
\item[(i)] $a'>a$ implies $(a,b+1) \not \in A' \times B'$ and thus $\big[X_n^{(a',b')}\big]_{(a,b+1)} =0$.
\item[(ii)] For $a'=a$ and $b'>b+1$ the conclusion is the same as above. Assuming $a'=a$ and $b'=b+1$  it follows that
$\big[X_n^{(a',b')}\big]_{(a,b+1)}=\big[X_4\big]_{(1,1)}=0$.
\end{itemize}
Therefore, entry $(a,b+1)$ of the left-hand side of (\ref{comblin}) is $\alpha_{a,b} \neq 0$ while the same entry on the right-hand side
of (\ref{comblin}) is 0, which is a contradiction.

Finally, we show that $\dim (\mathcal{E}_{\mathcal{Q}(n)} (-4)) \le (n-3)^2$ by showing that every element of the
subspace generated by $\mathcal{F}_n$ is completely determined by entries $x_{(i,j+1)}$ such that $(i,j) \in [n-3]^2$.

Let $S \subseteq [n]^2$ be the set of indexes $(p,q) \in [n]^2$ such that the entry $x_{(p,q)}$ of
$X \in \mathcal{E}_{\mathcal{Q}(n)} (-4)$ is completely determined by the entries $x_{(i,j+1)}$, with
$(i,j) \in [n-3]^2$. Clearly, $[n-3] \times ([n-2] \setminus \{ 1 \}) \subseteq S$. Since $x_{(1,1)} = x_{(n,1)} = 0$, it follows that
\begin{eqnarray*}
x_{(i,1)} &=& - \sum_{k=2}^i x_{(i+1-k,k)}, \ \text{for every $2 \leq i \leq n-2$,}\\
x_{(n-1,1)} &=& - \sum_{k=2}^{n-2} x_{(k,1)} \\
            &=& \hphantom{+} x_{(1,2)} \\
            & & + x_{(2,2)} + x_{(1,3)} \\
            & & \vdots \\
            & & + x_{(n-3,2)} + \cdots + x_{(2,n-3)} + x_{(1,n-2)} \\
            &=& \sum_{\substack{i,j \geq 1 \\ i+j \leq n-2}} x_{(i,j+1)}
\end{eqnarray*}
and then $[n] \times \{ 1 \} \subseteq S$. Additionally, since $x_{(1,n)} = x_{(n,n)} = 0$ it follows that
\begin{align*}
    & x_{(i,n-1)} = - \sum_{j=1}^{n-2} x_{(i,j)} - x_{(i,n)}, \ \text{for every $1 \leq i \leq n-3$,} \\
    & x_{(i+1,n)} = - \sum_{k=1}^{i} x_{(k,n-1-i+k)}, \ \text{for every $1 \leq i \leq n-3$,} \\
    & x_{(n-1,n)} = - \sum_{i=2}^{n-2} x_{(i,n)}, \quad  x_{(n-2,n-1)} = - \sum_{k=1}^{n-3} x_{(k,k+1)} - x_{(n-1,n)}, \\
    & x_{(n,n-1)} = - x_{(n-1,n)}, \quad x_{(n-1,n-1)} = - \sum_{i=2}^{n-2} x_{(i,n-1)} - x_{n,n-1}
\end{align*}

and then $[n] \times \{ n-1, n \} \subseteq S$. Finally, since for every $2 \le j \le n-2$
\begin{align*}
    & x_{(n,j)} = - \sum_{k=j}^{n-1} x_{(k,n+j-k)}, \\
    & x_{(n-2,j)} = - \sum_{k=1}^{j-1} x_{(n-2-j+k,k)} - x_{(n-1,j+1)} - x_{(n,j+2)},\\
    & x_{(n-1,j)} = - \sum_{i=1}^{n-2} x_{(i,j)} - x_{(n,j)},
\end{align*}
then $\{ n-2, n-1, n \} \times ([n-2] \setminus \{ 1 \} ) \subseteq S$.
\end{proof}

\begin{definition}\label{eigenvectors_n-4}
Given $n \in \mathbb{N}$, $i,j \in [n]$, $k \in \{ 2, \ldots, 2n \}$ and $\ell \in \{ -n+1, \ldots, n-1 \}$, let the row vector $R_i$, the column vector $C_j$, the sum vector $S_k$ and the difference vector $D_\ell$ be the vectors of $\mathbb{R}^{n \times n}$ defined as follows.
\begin{equation}
\begin{split}
R_i(x,y) & = \begin{cases}
              1, & \text{if } x=i \\
              0, & \text{otherwise.}
             \end{cases} \qquad C_j(x,y) = \begin{cases}
                                            1, & \text{if } y=j\\
                                            0, & \text{otherwise.}
                                           \end{cases} \\
S_k(x,y) & = \begin{cases}
              1, &\text{if } x+y=k\\
              0, &\text{otherwise.}
             \end{cases} \qquad  D_\ell(x,y) = \begin{cases}
                                             1, & \text{if } x-y=\ell\\
                                             0, & \text{otherwise.}
                                            \end{cases}.
\end{split}
\end{equation}
Furthermore, $\begin{cases}
                         Z^n = & \{D_0-S_{n+1}\} \text{ and } \\
                         Y^n = & \{Y_i^n = C_i+C_{n+1-i}-R_i-R_{n+1-i} \mid i \in \lceil \frac{n}{2}\rceil\}.
                         \end{cases}$
\end{definition}

Some examples for $n=3$ are the following.

\begin{table}[htbp!]
\centering
\begin{tabular}{|c|c|c|}\hline
\cellcolor{black}{\color{white}0}&0&\cellcolor{black}{\color{white}0}\\
0&\cellcolor{black}{\color{white}0}&0 \\
\cellcolor{black}{\color{white}1}&1&\cellcolor{black}{\color{white}1} \\ \hline
\end{tabular} \;\; \begin{tabular}{|c|c|c|}\hline
                   \cellcolor{black}{\color{white}0}&1&\cellcolor{black}{\color{white}0}\\
                    0&\cellcolor{black}{\color{white}1}&0 \\
                   \cellcolor{black}{\color{white}0}&1&\cellcolor{black}{\color{white}0} \\ \hline
                   \end{tabular} \;\; \begin{tabular}{|c|c|c|} \hline
                                      \cellcolor{black}{\color{white}0}&1&\cellcolor{black}{\color{white}0}\\
                                       1&\cellcolor{black}{\color{white}0}&0 \\
                                      \cellcolor{black}{\color{white}0}&0&\cellcolor{black}{\color{white}0} \\ \hline
                                      \end{tabular} \;\; \begin{tabular}{|c|c|c|}\hline
                                                         \cellcolor{black}{\color{white}1}&0&\cellcolor{black}{\color{white}0}\\
                                                          0&\cellcolor{black}{\color{white}1}&0 \\
                                                         \cellcolor{black}{\color{white}0}&0&\cellcolor{black}{\color{white}1} \\ \hline
                                                         \end{tabular}
\caption{$R_3$, $C_2$, $S_3$ and $D_0$.}
\end{table}

\begin{theorem}
\justifying
For $n \ge 3$, $n-4$ is an eigenvalue of $\mathcal{Q}(n)$ with multiplicity at least $\left\{\begin{array}{ll}
                                                                                          \frac{n-2}{2}, & \hbox{ if $n$ even};\\
                                                                                          \frac{n+1}{2}, & \hbox{otherwise.}
                                                                                         \end{array}\right.$
Furthermore, $Y^n$ and $Z^n$ are sets of linear independent eigenvectors of
$\mathcal{E}_{\mathcal{Q}(n)}\left(n-4\right)$.
\end{theorem}

\begin{proof}
In what follows we consider only an odd integer $n \ge 3$. The $n$ even case is similar with the exception that $Z^n_0 \not \in \mathcal{E}_{\mathcal{Q}(n)}(n-4)$ when $n$ is even.

First we note that $Y_k^n \in \mathcal{E}_{\mathcal{Q}(n)} \left(n-4\right)$ for every $k \in [\frac{n-1}{2}]$
and $Z^n_0 \in \mathcal{E}_{\mathcal{Q}(n)}(n-4)$. Since

\[
C_k (p,q) = R_k (q,p) \ \text{and} \ C_{n+1-k} (p,q) = R_{n+1-k} (q,p),
\]
we have
\[
\sum_{p+q=i+j} (Y_k^n)_{(p,q)} = 0.
\]
Similarly, since
\[
C_k (p,q) = R_{n+1-k} (n+1-q,n+1-p) \ \text{and} \ C_{n+1-k} (p,q) = R_k (n+1-q,n+1-p),
\]
we have
\[
\sum_{p-q=i-j} (Y_k^n)_{(p,q)} = 0.
\]

Hence
\begin{align*}
(A_{\mathcal{Q}(n)} Y_k^n)_{(i,j)} & = \sum_{(p,q) \in N(i,j)} (Y_k^n)_{(p,q)} \\
& = \sum_{p=i} (Y_k^n)_{(p,q)} + \sum_{q=j} (Y_k^n)_{(p,q)} + \\
& + \sum_{p+q=i+j} (Y_k^n)_{(p,q)} + \sum_{p-q=i-j} (Y_k^n)_{(p,q)} - 4 (Y_k^n)_{(i,j)} \\
& = \sum_{p=i} (Y_k^n)_{(p,q)} + \sum_{q=j} (Y_k^n)_{(p,q)} - 4 (Y_k^n)_{(i,j)} \\
& =
\begin{cases}
-(n-2)+(n-2)-0=0 & \text{if $i,j \in \{ k, n+1-k \}$}, \\
-(n-2)+2=-(n-4) & \text{if $i \in \{ k, n+1-k \}$ but $j \notin \{ k, n+1-k \}$}, \\
-2+(n-2)=n-4 & \text{if $j \in \{ k, n+1-k \}$ but $i \notin \{ k, n+1-k \}$}, \\
-2+2-0=0 & \text{if $i,j \notin \{ k, n+1-k \}$}
\end{cases} \\
& = (n-4) (Y_k^n)_{(p,q)}.
\end{align*}

Similarly,
\[
\sum_{p=i} (Z^n_0)_{(p,q)} = 0 = \sum_{q=j} (Z^n_0)_{(p,q)},
\]
and so
\begin{align*}
(A_{\mathcal{Q}(n)} Z^n_0)_{(i,j)} & = \sum_{(p,q) \in N(i,j)} (Z^n_0)_{(p,q)} \\
& = \sum_{p=i} (Z^n_0)_{(p,q)} + \sum_{q=j} (Z^n_0)_{(p,q)} + \\
& + \sum_{p+q=i+j} (Z^n_0)_{(p,q)} + \sum_{p-q=i-j} (Z^n_0)_{(p,q)} - 4 (Y_k^n)_{(i,j)}\\
& = \sum_{p+q=i+j} (Z^n_0)_{(p,q)} + \sum_{p-q=i-j} (Z^n_0)_{(p,q)} - 4 (Y_k^n)_{(i,j)}\\
& =
\begin{cases}
(n-1)-(n-1)-0=0 & \text{if $i = j = \frac{n+1}{2}$}, \\
1+(n-1)-4=n-4 & \text{if $i = j \neq \frac{n+1}{2}$}, \\
-(n-1)-1+4=-(n-4) & \text{if $i = n+1-j \neq \frac{n+1}{2}$}, \\
0+0-0=0 & \text{if $i \neq j$, $i+j \neq n+1$, and $i+j$ is odd}, \\
1-1-0=0 & \text{if $i \neq j$, $i+j \neq n+1$, and $i+j$ is even}
\end{cases} \\
& = (n-4) (Z^n_0)_{(p,q)}.
\end{align*}

Note that $\left\{ Y_1^n, Y_2^n, \ldots, Y_{\frac{n+1}{2}}^n \right\}$ is linearly dependent because
\[
Y_1^n + Y_2^n + \cdots + Y_{\frac{n-1}{2}}^n + \frac{1}{2} Y_{\frac{n+1}{2}}^n = 0.
\]

However $\left\{ Y_1^n, Y_2^n, \ldots, Y_{\frac{n-1}{2}}^n, Z^n_0 \right\}$ is linearly independent. For otherwise there would be $\alpha_1, \ldots, \alpha_{\frac{n-1}{2}}, \beta$ not simultaneously 0 such that
\[
L := \alpha_1 Y_1^n + \alpha_2 Y_2^n + \cdots + \alpha_{\frac{n-1}{2}} Y_{\frac{n-1}{2}}^n + \beta Z^n_0 = 0.
\]
Then $L_{(1,1)} = \beta = 0$ and so
\[
\alpha_1 Y_1^n + \alpha_2 Y_2^n + \cdots + \alpha_{\frac{n-1}{2}} Y_{\frac{n-1}{2}}^n = 0.
\]

Let $k \in [\frac{n-1}{2}]$ be the smallest integer such that $\alpha_k \neq 0$.
Then $L_{\left(\frac{n-1}{2},k\right)} = \alpha_k = 0$ contradicting our assumption.

Hence, $\dim \mathcal{E}_{\mathcal{Q}(n)} \geq \frac{n+1}{2}$.
\end{proof}

An interesting open problem related with the spectrum of $\mathcal{Q}(n)$ is the characterization of its integer eigenvalues. The eigenvalues of $\mathcal{Q}(n)$, computed for several values of $n$, suggest the following
conjecture.

\begin{conjecture}
For $n \ge 4$, the set of integer eigenvalues of the $n$-Queens' graph $\mathcal{Q}(n)$ is
\[
\sigma(\mathcal{Q}(n)) \cap \mathbb{Z} =
\begin{cases}
\left\{ -4, n-4 \right\}, & \text{if $n$ is even;} \\
\left\{ -4, -3, \ldots, \frac{n-11}{2} \right\} \cup \left\{ \frac{n-5}{2}, \ldots, n-5, n-4 \right\}, & \text{if $n$ is odd.}
\end{cases}
\]
Furthermore, when $n$ is even the eigenvalue $n-4$ has multiplicity $(n-2)/2$, and when $n$ is odd the eigenvalue $n-4$ has multiplicity $(n+1)/2$ and the eigenvalues $-3, -2, \dots, \frac{n-11}{2}, \frac{n-5}{2}, \dots, n-6, n-5$ are simple.
\end{conjecture}
From the computations we may say that this conjecture is true for $n \le 100$.  If this conjecture is true, then $\mathcal{Q}(n)$ has two distinct integer eigenvalues when $n$ is even and $n-1$ distinct integer eigenvalues when $n$ is odd.

\section{Equitable partitions}\label{sec4}
The eigenvalues $\mu$ of a graph $G$ which have an associated eigenspace ${\mathcal E}_G(\mu)$ not orthogonal to the all-one $n \times 1$ vector $\mathbf{j}$ are said to be {\it main}. The remaining eigenvalues are referred as {\it non-main}. The concept of main (non-main) eigenvalue was introduced in \cite{cds79} and further investigated in many papers since then. A recent overview was published in \cite{rowmain}. As consequence, from the proof of Theorem~\ref{th-4eigenvalue} and the definition of the family of eigenvectors $\mathcal{F}_n$ we may conclude that $-4$ is a non-main eigenvalue. The main eigenvalues of a graph $G$ are strictly related with the concept of \text{equitable partition} of the vertex set of $G$ and with the corresponding \textit{divisor (or quociente) matrix} $B$. The main eigenvalues of $G$ are also eigenvalues of
the divisor matrix of any equitable partition of $G$.

\begin{definition}[Equitable partition]
Given a graph $G$, the partition $V(G) = V_1 \dot{\cup} V_2 \dot{\cup} \dots \dot{\cup} V_k$ is an equitable partition if every
vertex in $V_i$ has the same number of neighbors in $V_j$, for all $i,j \in \{1, 2, \dots, k\}$. An equitable partition of $V(G)$
is also called equitable partition of $G$ and the vertex subsets $V_1, V_2, \dots, V_k$ are called the cells of the equitable partition.
\end{definition}
Every graph has a trivial equitable partition in which each cell is a singleton.
\begin{definition}[Divisor (or quotient) matrix]
Consider an equitable partition $\pi$ of $G$, $V(G) = V_1 \dot{\cup} V_2 \dot{\cup} \dots \dot{\cup} V_k$, where each vertex in $V_i$
has $b_{ij}$ neighbors in $V_j$ (for all $i,j \in \{1, 2, \dots, k\}$). The matrix $B_{\pi} = \left(b_{ij}\right)$ is called
the divisor (or quociente) matrix of $\pi$.
\end{definition}
The following results for graphs with equitable partitions are fundamental.
\begin{lemma}\cite{CP84,rowmain}\label{integer_coeficcients}
If $G$ is a graph with $p$ distinct main eigenvalues $\mu_{j_1}, \ldots, \mu_{j_p}$, then the main characteristic polynomial of $G$,
\begin{equation}
M_G(x) = \prod^{p}_{i =1}( x - \mu_{j_i}) = x^{p} - c_1 x^{p-1} - c_2 x^{p-2} - \cdots - c_{p-1}x - c_{p}, \label{mcpG}\\
\end{equation}
has integer coefficients.
\end{lemma}

\begin{theorem}\cite{CRS10}\label{CRS_Th}
Let $G$ be a graph with adjacency matrix $A$ and let $\pi$ be a partition of $V(G)$ with characteristic matrix $C$.
\begin{enumerate}
\item If $\pi$ is equitable, with divisor matrix $B_{\pi}$, then $AC = CB_{\pi}$. \label{CRS_Th_1}
\item The partition $\pi$ is equitable if and only if the column space of $C$ is $A$-invariant. \label{CRS_Th_2}
\item The characteristic polynomial of any divisor matrix of $G$ divides the characteristic polynomial of $G$ and it is divisible by $M_G(x)$. \label{CRS_Th_3}
\end{enumerate}
\end{theorem}

Since the largest eigenvalue of any connected graph $G$ with $n \ge 2$ vertices is a main eigenvalue (see \cite[Th. 03 and 04]{cds79}), an immediate consequence of Theorem~\ref{CRS_Th}-\ref{CRS_Th_3} is that the largest eigenvalue of $G$ is an eigenvalue of any of its divisor matrices.\\

A very interesting property of the $n$-Queens' graphs consists that they admit equitable partitions with
divisor matrices of order $\frac{(\lceil\frac{n}{2} \rceil+1)\lceil \frac{n}{2} \rceil}{2}$, for $n \ge 3$.
These equitable partitions can be obtained applying the following algorithm, where the squares of the chessboard $\mathcal{T}_n$,
corresponding to the vertices of $\mathcal{Q}(n)$, are labeled with the numbers of the cells they belong. Therefore, the squares
belonging to the same cell have the same number.

\begin{itemize}
\item[{\bf Algorithm}]\textbf{1}[for the determination of an equitable partition of $\mathcal{Q}(n)$]
\item[{\bf Require}:] The $n \times n$ chess board, $\mathcal{T}_n$.
\item[{\bf Ensure}:]  The labels of the squares of $\mathcal{T}_n$ which are the numbers of the cells of an equitable partition.
\item[{\bf Part I}:] Labeling procedure.
      \begin{itemize}
      \item[$1$] Assign the the number \textbf{$0$} to the first square (the top left square).
      \item[$2$] Assign the numbers \textbf{$1$} and \textbf{$2$} to the first and second squares of the second column (from the top to bottom), respectively.
      \item[$\;$] $\vdots$\\
      \item[$\lceil \frac{n}{2} \rceil$] Assign the numbers $ \sum_{i=0}^{\lceil \frac{n}{2} \rceil-1}{i}+1, \dots, \frac{(\lceil\frac{n}{2} \rceil+1)\lceil \frac{n}{2} \rceil}{2}-1$
                                         to the first $\lceil \frac{n}{2} \rceil$ squares of the $\lceil \frac{n}{2} \rceil$-th column (from top to bottom).
      \end{itemize}
\item[{\bf Part II}:] Geometric reflection procedure.
      \begin{enumerate}
      \item Reflect the triangle formed by the squares labeled by the above steps, using the vertical cathetus of the triangle as mirror line (after this reflection, we have two right
            triangles sharing the same vertical line).
      \item Reflect both triangles, each one using its hypotenuse as mirror line (after this reflections all the squares in the top $\lceil \frac{n}{2} \rceil$ lines are labeled).
      \item Reflect the rectangle formed by the the upper $\lfloor \frac{n}{2} \rfloor$ lines taking the horizontal middle line of the chessboard as mirror line (after this reflection
            all the squares have labels which are the numbers of their cells).
      \end{enumerate}
\item[$\;$] \hspace{-1.5cm}{\bf End Algorithm}.
\end{itemize}

\begin{example}
Let us apply the Algorithm~1 to the determination of an equitable partition of the vertices of $\mathcal{Q}(6)$, corresponding
to the squares in $\mathcal{T}_6$, the divisor matrix $B$ and its characteristic polynomial.

\begin{table}[h!]
			\centering
{\footnotesize \begin{tabular}{|c|c|c|c|c|c|}\hline
\textbf{0} & \textbf{1} & \textbf{3} & $\;$ & $\;$ & \\ \hline
           & \textbf{2} & \textbf{4} &      &      & \\ \hline
		   &            & \textbf{5} &      &      & \\ \hline
		   &            &            &      &      & \\ \hline
           &            &            &      &      & \\ \hline
           &            &            &      &      & \\ \hline
\end{tabular}
\qquad
\begin{tabular}{|c|c|c|c|c|c|}\hline
\textbf{0} & \textbf{1} & \textbf{3} & $3$ & $1$ & $0$\\ \hline
           & \textbf{2} & \textbf{4} & $4$ & $2$ & \\ \hline
		   &            & \textbf{5} & $5$ &     & \\ \hline
		   &            &            &     &     & \\ \hline
           &            &            &     &     & \\ \hline
           &            &            &     &     & \\ \hline
\end{tabular}
\qquad
\begin{tabular}{|c|c|c|c|c|c|}\hline
\textbf{0} & \textbf{1} & \textbf{3} & $3$ & $1$ & $0$\\ \hline
$1$        & \textbf{2} & \textbf{4} & $4$ & $2$ & $1$\\ \hline
$3$		   & $4$        & \textbf{5} & $5$ & $4$ & $3$\\ \hline
		   &            &            &     &     & \\ \hline
           &            &            &     &     & \\ \hline
           &            &            &     &     & \\ \hline
\end{tabular}

\medskip

\begin{tabular}{|c|c|c|c|c|c|}\hline
\textbf{0} & \textbf{1} & \textbf{3} & $3$ & $1$ & $0$ \\ \hline
$1$        & \textbf{2} & \textbf{4} & $4$ & $2$ & $1$ \\ \hline
$3$        & $4$        & \textbf{5} & $5$ & $4$ & $3$ \\ \hline
$3$        & $4$        & $5$        & $5$ & $4$ & $3$ \\ \hline
$1$        & $2$        & $4$        & $4$ & $2$ & $1$ \\ \hline
$0$        & $1$        & $3$        & $3$ & $1$ & $0$ \\ \hline
\end{tabular}}
\caption{Labeling of the squares of $\mathcal{T}_n$ by application of Algorithm~1.}\label{equitable_partition}
\end{table}

\medskip

\noindent The divisor matrix of the equitable partition with cells numbered from $0$ to $5$ presented in Table~\ref{equitable_partition} is the matrix
\[
B = \bordermatrix{  & 0 & 1 & 2 & 3 & 4 & 5 \cr
                  0 & 3 & 4 & 2 & 4 & 0 & 2 \cr
                  1 & 2 & 4 & 2 & 2 & 4 & 1 \cr
                  2 & 2 & 4 & 3 & 2 & 4 & 2 \cr
                  3 & 2 & 2 & 1 & 4 & 4 & 2 \cr
                  4 & 0 & 4 & 2 & 4 & 4 & 3 \cr
                  5 & 2 & 2 & 2 & 4 & 6 & 3}.
\]
The characteristic polynomial of $B$ is the polynomial
\[
p_B(x) = x^6 - 21 x^5 + 73 x^4 + 109 x^3 - 686 x^2 + 580 x - 8 .
\]
\end{example}
According to Theorem~\ref{CRS_Th}-\ref{CRS_Th_3}, the eigenvalues of the divisor matrix $B$ obtained in the above
example, which are the ones belonging to
\[
\sigma(B)=\{ 16.24, 3.63, 2.85, 1.17, 0.01, -2.91 \},
\]
are also eigenvalues of $\mathcal{Q}(6)$ and the largest eigenvalue of $B$ is also the largest eigenvalue of
$\mathcal{Q}(6)$. Moreover, taking into account Lemma~\ref{integer_coeficcients}, since $p_B(x)$ is irreducible
in $\mathbb{Z}[x]$, we may conclude $p_B(x) = M_{\mathcal{Q}(6)}(x)$.

\begin{theorem}
Every $n$-Queens' graph, $\mathcal{Q}(n)$, with $n \ge 3$, has an equitable partition with
$\frac{\left(\lceil \frac{n}{2} \rceil+1\right)\lceil \frac{n}{2} \rceil}{2}$
cells, determined by Algorithm~1.
\end{theorem}

\begin{proof}
The application of Algorithm~1 produces a partition $\pi$ of the vertices (squares) of $\mathcal{Q}(n)$ ($\mathcal{T}_n$) into $\frac{\left(\lceil \frac{n}{2} \rceil+1\right)\lceil \frac{n}{2} \rceil}{2}$
cells. Furthermore, from Algorithm~1, we may conclude that each cell of $\pi$ has $8$, $4$ or $1$ vertices (squares) and, by symmetry, each cell induces a regular graph and each vertex in a particular cell
has the same number of neighbors in any other cell.
\end{proof}

\medskip\textbf{Acknowledgments.}
This work is supported by the Center for Research and Development in Mathematics and Applications (CIDMA) through the Portuguese Foundation for Science and Technology (FCT - Fundação para a Ciência e a Tecnologia), reference UIDB/04106/2020. I.S.C. also thanks the support of FCT - Fundação para a Ciência e a Tecnologia via the Ph.D. Scholarship PD/BD/150538/2019.

\end{document}